\documentclass[reqno, 11pt]{amsart}
\usepackage{amsmath}
\usepackage{framed,comment,enumerate}
\usepackage[pdftex]{graphicx}
\usepackage{epstopdf}
\usepackage{xcolor, framed}
\usepackage{color}

\usepackage{amsmath, amsthm, amssymb, esint}
\usepackage{amsfonts}
\usepackage{bbm}
\usepackage[ansinew]{inputenc}
\usepackage[dvips]{epsfig}
\usepackage{graphicx}
\usepackage[english]{babel}
\usepackage{enumerate}
\usepackage{enumitem}
\usepackage{hyperref}
\usepackage{fbox}

\def\R {\mathbb{R}}
\def\N{\mathbb{N}}
\def\eps{\varepsilon}
\def\Sph{\mathbb{S}^{d-1}}
\def\DimH{\mathrm{dim}_{\mathcal{H}}}
\def\HausM{\mathcal{H}^{d-1}}

\def\Epi{\mathrm{Epi}}
\def\Hyp{\mathrm{Hyp}}

\def\En{\mathcal{E}}
\def\EnAC{\mathcal{E}_{AC}^+}
\def\EnAP{\mathcal{E}_{AP}^+}

\def\FamPhi{\{\varphi_t\}_{t\in(-1,1)}}

\def\MPhiT{\mathcal{M}[\varphi_t]}

\def\HOne{H^1(B_1)}
\def\Diri{|\nabla u|^2}

\def\Dac{d^*_{AC}}
\def\Dap{d^*_{AP}}

\def\PosS{\{u>0\}}

\def\NegS{\{u<0\}}
\def\ConS{\{u=0\}}

\def\M{\mathcal{M}}

\def\GU{\Gamma(u)}
\def\Sing{\mathrm{Sing}}
\def\SingU{\mathrm{Sing}(u)}
\def\SingUT{\mathrm{Sing}(u_t)}
\def\RedGU{\Gamma^*(u)}
\def\RedSing{\mathcal{S}^*}

\def\hem{\hspace{0.5em}}
\def\vem{\vspace{0.6em}}

\newtheorem{thm}{Theorem}[section]
\newtheorem{prop}[thm]{Proposition} 

\newtheorem{ass}{Assumption} 

\newtheorem{cor}[thm]{Corollary}
\newtheorem{lem}[thm]{Lemma}
\theoremstyle{definition}
\newtheorem{defi}[thm]{Definition}
\numberwithin{equation}{section}

\theoremstyle{remark}
\newtheorem{rem}[thm]{Remark}

\title[Generic properties]{Generic properties in free boundary problems} 

\author{Xavier Fern\'andez-Real}
\address{Institute of Mathematics, \'Ecole Polytechnique F\'ed\'erale de Lausanne, Lausanne, Switzerland}
\email{xavier.fernandez-real@epfl.ch}

\author{Hui Yu}
\address{Department of Mathematics,	National University of Singapore, Singapore}
\email{huiyu@nus.edu.sg}



\thanks{X. F. was supported by the Swiss National Science Foundation (SNF grants 200021\_182565 and PZ00P2\_208930),   by the Swiss State Secretariat for Education, Research and Innovation (SERI) under contract number MB22.00034, and by the AEI project PID2021-125021NA-I00 (Spain).}

\begin{document}

\begin{abstract}
In this work, we show the generic uniqueness of minimizers for a large class of energies, including the Alt-Caffarelli and Alt-Phillips functionals.

We then prove the generic regularity of free boundaries for minimizers of the one-phase Alt-Caffarelli and Alt-Phillips functionals, for a monotone family of boundary data $\FamPhi$. More precisely, we show that for a  co-countable subset of $\FamPhi$, minimizers have smooth free boundaries in $\R^5$ for the Alt-Caffarelli and in $\R^3$ for the Alt-Phillips functional. In general dimensions, we show that  the singular set is one dimension smaller than expected  for almost every boundary datum in $\FamPhi$.
\end{abstract}

\maketitle
\section{Introduction}
In this work, we study minimizers of the following energy in free boundary problems
\begin{equation}
\label{EqnFreeBoundaryEnergy}
\En_{\pm}(u):=\int_{B_1}\left(|\nabla u|^2+F_+(u^+)\chi_{\PosS}+F_-(u^-)\chi_{\{u<0\}}\right),
\end{equation}
 where $a^+ = {\rm max}\{a, 0\}$ and $a^- = {\rm max}\{-a, 0\}$  denote the positive and  negative parts of a real number $a$. For the nonlinearities $F_\pm\in C([0,+\infty))\cap C^1((0,+\infty))$, we assume that
\begin{equation}
\label{EqnConditionOnFpm}
 F_\pm\ge0, \hem F_\pm'\ge 0\hem \text{ and }\hem F_\pm \text{ are concave on }(0,+\infty).
\end{equation} 
Among this family of energies, two examples are the \textit{Alt-Caffarelli functional}
$$
\En_{AC}(u):=\int_{B_1}\left(\Diri+\lambda_+\chi_{\PosS}+\lambda_-\chi_{\NegS}\right),
$$
and the  \textit{Alt-Phillips functional}
$$
\En_{AP}(u):=\int_{B_1}\left(\Diri+\lambda_+(u^+)^\gamma+\lambda_-(u^-)^\gamma\right).
$$
Here $\lambda_{\pm}$ are non-negative constants, and $\gamma$ is a parameter in $(0,1)$.\footnote{The Alt-Phillips functional leads to interesting free boundary problems for $\gamma\in(0,2)$. 
However, $\gamma\in(0,1)$ is the range relevant to the study of uniqueness of minimizers. For $\gamma\ge1$, the Alt-Phillips functional becomes convex, and hence, minimizers are unique. 

When $\gamma\ge1$, the regularity of free boundaries  has a different flavor. See, for instance,  \cite{B} and   \cite{WY}.}
In this context, the \textit{free boundary} refers to the set
\begin{equation}
\label{EqnFreeBoundary}
\Gamma(u):=\partial\{u\neq0\}.
\end{equation}

These functionals received intense attention in the past few decades. For classical results, the readers may consult \cite{AC, ACF, AP,P}. For interesting recent developments, we refer to \cite{D1, DJS, DS, EdSV, EFeY, ESV, JS}. On the one-phase Alt-Caffarelli functional, two wonderful resources are the monograph \cite{CS} and the lecture notes \cite{V}.

Under reasonable assumptions, the existence of a minimizer follows by standard methods. Much subtler are questions concerning the uniqueness of the minimizer and the regularity of the free boundary. Even for smooth boundary data, since the functionals are non-convex, the minimizer might not be unique; and singularities on free boundaries are in general inevitable. 

In this work, we show that \textit{generically these issues do not occur}. In particular, under small perturbations of the boundary data,  minimizers are unique, and the free boundary is smooth in one dimension more than expected. 

\subsection{Generic uniqueness}
 In this work, we discuss several different energies. If $\En(\cdot)$ is one of them, we denote  the sets of minimizers  with boundary datum $\varphi\in H^1 (B_1)$ by
\begin{equation}
\label{EqnSetOfMin}
 \mathcal{M}[\En,\varphi]:=\{u\in H^1_0(B_1) + \varphi : \hem \En(u)\le \En(v) \text{ for all }v\in H^1_0(B_1)+\varphi\}.\footnote{When there is no ambiguity, we omit $\En$ from the notation. When it is unnecessary to mention the boundary data, we write $u\in\mathcal{M}[u]$ to say that $u$ is a minimizer with respect to its boundary value. }
\end{equation}

Under our assumption  \eqref{EqnConditionOnFpm}, it is not difficult to show that $\mathcal{M}[\En_\pm,\varphi]$ is non-empty.  In general,   the set $\mathcal{M}[\En_\pm,\varphi]$ might contain multiple elements. Accommodating more than one minimizer, however, imposes extra restrictions on the boundary data. As a result, the  non-uniqueness of minimizers should not happen frequently. For a family of boundary data $\FamPhi$, we should expect that for most of $\FamPhi$ the set $\mathcal{M}[\varphi_t]$ is a singleton. 

To make this precise,  we impose the following conditions on the family $\FamPhi$:
\begin{ass}[Assumptions on boundary data for generic uniqueness]
\label{AssGU}
We assume that 
\begin{enumerate}
\item{Each $\varphi_t$ is Lipschitz continuous;}
\item{For any $t\ge s$, we have 
$$
\varphi_t\ge\varphi_s \text{ on $\partial B_1$};
$$  }
\item{\label{it:333} For any $t>s$, and for any connected component $\mathcal{C}$ of the sets $\{\varphi_s>0\}\cap\partial B_1$ or $\{\varphi_t < 0\}\cap\partial B_1$, we have 
$$
\varphi_t(x_0)>\varphi_s(x_0) \text{ for some $x_0\in\mathcal{C}$;}
$$}
\end{enumerate}
\end{ass}

\begin{rem}
\label{RemAssGU}
The assumptions are qualitative only.  The Lipschitz assumption is used to show the continuity of minimizers  in $\overline{B_1}$. For this purpose, weaker assumptions could suffice, but we did not pursue this direction for simplicity. Moreover, we allow the data to change sign and to have zero patches on $\partial B_1$. Such boundary data are meaningful in  two-phase free boundary problems when the free boundary touches the fixed boundary. 

We remark that the role of condition  \eqref{it:333} is   necessary, to avoid situations in which two disconnected components do not see each other. 
\end{rem} 

Under these assumptions, our result on the generic uniqueness of minimizers reads as follows (recall \eqref{EqnSetOfMin}).  Such a generic uniqueness type result is new even for the one-phase Alt-Caffarelli functional.
\begin{thm}
\label{ThmGU}
  Let $\En_{\pm}$ be the energy in \eqref{EqnFreeBoundaryEnergy}-\eqref{EqnConditionOnFpm}, and let $\FamPhi$ satisfy Assumption~\ref{AssGU}. Then, there is a countable set $I\subset(-1,1)$ such that 
$
\mathcal{M}[\En_{\pm}, \varphi_t]$ is a singleton for all $t\in (-1,1)\backslash I$. 
\end{thm} 

As a corollary,  the non-uniqueness of minimizers can be perturbed away:
\begin{cor}
\label{CorGU}
  Let $\En_{\pm}$ be the energy in \eqref{EqnFreeBoundaryEnergy}-\eqref{EqnConditionOnFpm}, and let $\varphi\in H^{1/2}(\partial B_1)$ (resp. $\varphi\in C(\partial B_1)$).  Then, for any $\eps>0$ we can find $\tilde{\varphi}$ such that 
$$
\|\tilde{\varphi}-\varphi\|_{H^{1/2}(\partial B_1)}<\eps\quad\left(\text{resp.}\ \|\tilde{\varphi}-\varphi\|_{L^\infty(\partial B_1)}<\eps\right),\qquad\text{and}\qquad \mathcal{M}[\En_{\pm}, \tilde{\varphi}] \text{ is a singleton}.
$$
\end{cor}

When the nonlinearities $F_\pm$ are not convex, the Euler-Lagrange equations corresponding to the energy  \eqref{EqnFreeBoundaryEnergy} do not enjoy the comparison principle. Minimizers are in general not unique. Nevertheless, we are able to show that under our assumptions, the minimizers are ordered. That is, suppose that $u_t\in\mathcal{M}[\varphi_t]$ and $u_s\in\mathcal{M}[\varphi_s]$ for $t>s$, then 
\begin{equation}
\label{EqnComparisonBetweenMinimizers}
u_t\ge u_s \text{ in }B_1,
\end{equation}
(see Proposition~\ref{PropComparison}). For this comparison, our method is inspired by De Silva-Jerison-Shahgholian \cite{DJS} and Edelen-Spolaor-Velichkov \cite{ESV}, where the authors study minimizers lying on one side of a cone for the one-phase Alt-Caffarelli functional. 

The proof of Theorem~\ref{ThmGU} is now the following: suppose that $\mathcal{M}[\varphi_t]$ is not a singleton, so that we  find distinct $\overline{u}_t, \underline{u}_t\in\MPhiT$. We can then squeeze a non-trivial ball $\mathcal{B}(t)\subset B_1\times\R$ between the graphs of $\overline{u}_t$ and $\underline{u}_t$. Denote by $I$ the collection of $t$ for which this happens. Then, the comparison in \eqref{EqnComparisonBetweenMinimizers} guarantees 
$$
\mathcal{B}(t)\cap\mathcal{B}(s)=\emptyset \hem \text{ for $t,s\in I$ and $t\neq s$.}
$$ 
This implies that $I$ is  countable, since the number of non-overlapping balls in Euclidean spaces are countable.

\subsection{Generic regularity of the free boundary}
Regarding the regularity of the free boundary, not much is known for general energies of the form \eqref{EqnFreeBoundaryEnergy}. We focus on two models: the \textit{one-phase Alt-Caffarelli functional}
\begin{equation}
\label{EqnAC}
\En_{AC}^+(u):=\int_{B_1}\left(\Diri+\chi_{\PosS}\right),
\end{equation} 
and the \textit{one-phase Alt-Phillips functional}
\begin{equation}
\label{EqnAP}
\En_{AP}^+(u):=\int_{B_1}\left( \Diri+u^\gamma\right).
\end{equation}
Here we assume $u\ge0$  and $\gamma\in(0,1)$.

\subsubsection{The Alt-Caffarelli functional}

The energy in \eqref{EqnAC} was first studied by Alt-Caffarelli \cite{AC}. For a minimizer $u$,  they addressed its regularity and nondegeneracy, which imply that the positive set $\PosS$ is locally a set of finite perimeter (see \cite{M} for a monograph on sets of finite perimeter).

Consequently,  the free boundary (see \eqref{EqnFreeBoundary}) decomposes into 
\begin{equation}
\label{EqnFBDecomposition}
\Gamma(u)=\Gamma^*(u)\cup \SingU,
\end{equation}
where $\Gamma^*(u)$ denotes the reduced boundary, and $\SingU$ is the remaining \textit{singular part} of the free boundary; and the reduced boundary $\Gamma^*(u)$ is smooth \cite{AC} (the same result was obtained for viscosity solutions by De Silva \cite{D1}).

After this, the task is to  estimate the size of $\SingU$, for instance, in terms of its Hausdorff dimension. With the monotonicity formula by Weiss \cite{W}, it suffices to estimate the following
\begin{equation}
\label{EqnCriticalAC}
d^*_{AC}:=\max\{d\in\N:\hem \text{homogeneous minimizers are rotations of } (x_d)^+ \text{ in }\R^d\}\footnote{While its exact value remains unknown, we currently know that $d_{AC}^*\ge 4$ \cite{CJK,JS} and that $d^*_{AC}\le 6$ \cite{DJ}.}.
\end{equation} 
For a minimizer $u$ of $\EnAC(\cdot)$ in $B_1\subset\R^d$, it is known that $\SingU=\emptyset$ if $d\le\Dac$, and that $\SingU$ is locally discrete if $d= \Dac+1$. For general $d$, the Hausdorff dimension of $\SingU$ can be estimated as 
$$
\DimH(\SingU)\le d-\Dac-1.
$$
(See \cite{W, V}.)

We show that generically, these estimates can be improved by $1$. To be precise, we study minimizers with respect to a one-parameter family of boundary data $\FamPhi$, under the following assumptions: 
\begin{ass}[Assumptions on boundary data for generic regularity]
\label{AssGR}
For the family $\FamPhi$, we assume the following
\begin{enumerate}
\item{For all $t\in(-1,1)$, we have $\varphi_t\ge0$ in $ B_1$;}
\item{\label{it:33_2}There is $L<+\infty$ such that for all $t\in(-1,1)$, we have
$$
\varphi_t+|\nabla\varphi_t|\le L \text{ in } B_1;
$$}
\item{\label{it:44_2}There is $\theta>0$ such that for all $-1 < s < t < 1$, we have
$$
\varphi_t-\varphi_s\ge\theta(t-s) \text{ in }  B_1.
$$}
\end{enumerate}
\end{ass}
\noindent Assumptions \eqref{it:33_2} and \eqref{it:44_2} are quantified versions of their counterparts in Assumption \ref{AssGU}.
\begin{rem}\label{RemContAlmEverywhere}
Observe that, since the map $t\mapsto  \varphi_t$ is monotone increasing and $\varphi_t$ is uniformly Lipschitz for $t\in (-1, 1)$, then $t\mapsto \varphi_t$ is continuous for all $t\in (-1, 1)\setminus I_D$, where $I_D$ is countable.
\end{rem}

Our result on the generic regularity of the free boundary for the Alt-Caffarelli functional  reads:  
\begin{thm}
\label{ThmGRAC}
For each $t\in(-1,1)$, suppose that $u_t\in\mathcal{M}[\En_{AC}^+,\varphi_t]$ for a family $\FamPhi$ satisfying Assumption \ref{AssGR}.

For $d=\Dac+1$, there is a countable subset $J\subset(-1,1)$ such that
$$
\SingUT=\emptyset\hem \text{ for all } t\in(-1,1)\backslash J.
$$

For $d\ge\Dac+2$, we have
$$
\DimH(\SingUT)\le d-\Dac-2 \hem\text{ for almost every }t\in(-1,1).
$$
\end{thm} 
Recall the parameter $\Dac$ from \eqref{EqnCriticalAP}, the Alt-Caffarelli functional $\En_{AC}^+$ from \eqref{EqnAC}, the singular set $\SingU$ from \eqref{EqnFBDecomposition}, and the set of energy minimizers $\mathcal{M}[\cdot,\cdot]$ from \eqref{EqnSetOfMin}.

As a corollary, singularities on free boundaries can be perturbed away.
\begin{cor}
\label{CorGRAC}
  Let   $\varphi\in H^{1/2}(\partial B_1)$ (resp. $\varphi\in C(\partial B_1)$), $\varphi\ge 0$ on $\partial B_1$.  Then, for any $\eps>0$ we can find $\tilde{\varphi}$ such that 
$$
\|\tilde{\varphi}-\varphi\|_{H^{1/2}(\partial B_1)}<\eps\quad\left(\text{resp.}\ \|\tilde{\varphi}-\varphi\|_{L^\infty(\partial B_1)}<\eps\right),
$$
and for $u\in \mathcal{M}[\En_{AC}^+, \tilde\varphi]$,
$$
\begin{array}{rcll}
\SingU&=&\emptyset\hem   &  \text{ if }d=\Dac+1, \\
\DimH(\SingU)&\le& d-\Dac-2   & \text{ if }d\ge\Dac+2.
\end{array}
$$
\end{cor}

In particular, we have that for almost every solution, the free boundary is smooth in~$\R^d$ for $d\le 5$.  

\subsubsection{The Alt-Phillips functional} The energy $\EnAP(\cdot)$ in \eqref{EqnAP} was first studied by Phillips \cite{P}. In Alt-Phillips \cite{AP}, the authors established the regularity and nondegeneracy of a minimizer $u$. This allows the decomposition of the free boundary from \eqref{EqnFreeBoundary} as
$$
\GU=\RedGU\cup\SingU,
$$
where $\RedGU$ is the reduced boundary and is $C^{1,\alpha}$ \cite{AP}, and $\SingU$ is the  singular part of the free boundary. Similar results were obtained for viscosity solutions by De Silva-Savin \cite{DS}.

Due to the nonlinearity, much less is known about the singular part in this context. Nevertheless, to estimate the dimension of $\SingU$, it suffices to estimate the following
\begin{equation}
\label{EqnCriticalAP}
\Dap:=\max\{d\in\N:\hem \text{homogeneous minimizers are rotations of } [(x_d/\beta)^+]^\beta \text{ in }\R^d\}\footnote{While we know $\Dap\ge 2$ \cite{AP}, the exact value of $\Dap$ remains unknown.},
\end{equation} 
where $\beta=\frac{2}{2-\gamma}$.

For a minimizer $u$ of $\EnAP(\cdot)$ in $B_1\subset\R^d$, we show the following:
\begin{thm}
\label{ThmSingAP}
Suppose that $u\in\mathcal{M}[\EnAP,u]$ in $B_1\subset\R^d$, then we have
\begin{enumerate}
\item{ If $d\le\Dap$, then $\SingU=\emptyset$;}
\item{ If $d=\Dap+1$, then $\SingU$ is locally discrete;}
\item{ For $d\ge\Dap+2$, $\DimH(\SingU)\le d-\Dap-1$.}
\end{enumerate}
\end{thm} 
Analogous to their counterparts for the Alt-Caffarelli functional, these estimates have not yet been rigorously proved in the literature. We give a proof in the Appendix \ref{AppSingAP}.

Our result on the generic regularity of the free boundary states that, generically, these estimates can be improved by $1$:
\begin{thm}
\label{ThmGRAP}
For each $t\in(-1,1)$, suppose that $u_t\in\mathcal{M}[\En_{AP}^+,\varphi_t]$ for a family $\FamPhi$ satisfying Assumption \ref{AssGR}.

For $d=\Dap+1$, there is a countable subset $J\subset(-1,1)$ such that
$$
\SingUT=\emptyset \hem \text{ for all } t\in(-1,1)\backslash J.
$$

For $d\ge\Dap+2$, we have
$$
\DimH(\SingUT)\le d-\Dap-2\hem \text{ for almost every }t\in(-1,1).
$$
\end{thm} 
Recall the parameter $\Dap$ from \eqref{EqnCriticalAP}, the Alt-Phillips functional $\En_{AP}^+$ from \eqref{EqnAP}, the singular set $\SingU$ from \eqref{EqnFBDecomposition}, and the set of energy minimizers $\mathcal{M}[\cdot,\cdot]$ from \eqref{EqnSetOfMin}.

As a corollary, singularities on free boundaries can be perturbed away.
\begin{cor}
\label{CorGRAP}
  Let   $\varphi\in H^{1/2}(\partial B_1)$ (resp. $\varphi\in C(\partial B_1)$), $\varphi\ge 0$ on $\partial B_1$.  Then, for any $\eps>0$ we can find $\tilde{\varphi}$ such that 
$$
\|\tilde{\varphi}-\varphi\|_{H^{1/2}(\partial B_1)}<\eps\quad\left(\text{resp.}\ \|\tilde{\varphi}-\varphi\|_{L^\infty(\partial B_1)}<\eps\right),
$$
and for $u\in \mathcal{M}[\En_{AP}^+, \tilde\varphi]$,
$$
\begin{array}{rcll}
\SingU&=&\emptyset\hem   &  \text{ if }d=\Dap+1, \\
\DimH(\SingU)&\le& d-\Dap-2   & \text{ if }d\ge\Dap+2.
\end{array}
$$
\end{cor}
\noindent In particular, in the physical space $\R^3$, almost every solution has a regular free boundary.

\vem

The proofs of Theorems~\ref{ThmGRAC} and \ref{ThmGRAP} are inspired by their counterparts for the obstacle problem  done by Monneau \cite{Mo} and Figalli, Ros-Oton, and Serra, \cite{FRS}, and are based on dimension reduction and cleaning lemmas. In particular, we can show that the size of the singular set is the same for a single solution and for the whole family $\bigcup_{t\in (-1, 1)}\Sing(u_t)$.

Apart from the works mentioned, similar results have been obtained by the first author with Ros-Oton in \cite{FeR} and with Torres-Latorre in \cite{FeT}, in the context of the fractional or thin obstacle problem; and by Hardt-Simon \cite{HS} and  Chodosh-Mantoulidis-Schulze \cite{CMS1,CMS2} for minimal surfaces. 
Compared with these, however, much less is known about minimizing cones for our problems \eqref{EqnAC} and \eqref{EqnAP}. As a result, new ideas are required for Theorem \ref{ThmGRAC} and Theorem \ref{ThmGRAP}. 

One crucial step is to rule out minimizing cones that are monotone. For the Alt-Caffarelli functional \eqref{EqnAC}, this is achieved with the theory on monotone solutions \cite{DJ, EFeY}. Unfortunately, this theory is missing for the Alt-Phillips functional \eqref{EqnAP}. Instead, for a minimizer $u$ of the Alt-Phillips functional, we show that if $u$ is monotone along $(d-\Dap)$ independent directions, then $\SingU=\emptyset$. This is obtained by a detailed expansion of the minimizer around a point in $\RedGU$ and an improvement-of-monotonicity argument.

\subsection{Organization of the paper}This work is structured as follows. In the next section, we collect some preliminary properties about the functionals. In Section~\ref{sec:3}, we prove the generic uniqueness of minimizers from Theorem \ref{ThmGU}. In the remaining sections, we study the generic regularity of free boundaries from Theorems~\ref{ThmGRAC} and~\ref{ThmGRAP}. In Section~\ref{SecGRAC}, we establish the result for the Alt-Caffarelli functional. In Section~\ref{sec:5}, we establish the result for the Alt-Phillips functional. On top of these, there are two appendices. In Appendix~\ref{AppSingAP}, we establish Theorem \ref{ThmSingAP}, concerning the size of the singular set in the Alt-Phillips problem. In Appendix~\ref{AppExpansionNearRegularPoint}, we quantify an expansion of the minimizer to the Alt-Phillips function, originally given by De Silva-Savin \cite{DS}.

\section{Preliminaries}
This preliminary section contains four subsections. In the first, we discuss properties of general energy in  \eqref{EqnFreeBoundaryEnergy}. In the next two subsections, we collect results for the one-phase Alt-Caffarelli functional \eqref{EqnAC} and the one-phase Alt-Phillips functional \eqref{EqnAP}, respectively. We conclude this section with some tools from the geometric measure theory. 

\subsection{Properties of general energies}
Recall the energy $\En_\pm$ from  \eqref{EqnFreeBoundaryEnergy}.  We assume that the nonlinearities $F_\pm$ satisfy \eqref{EqnConditionOnFpm}. 


For the set of minimizers $\mathcal{M}[\cdot,\cdot]$ defined in \eqref{EqnSetOfMin}, it is not difficult to see that  $u\in\mathcal{M}[\En_\pm,u]$ satisfies
\begin{equation}
\label{EqnSecondEL}
\Delta u=F'_+(u)/2 \text{ in }\PosS, \text{ and }\Delta u=-F'_-(u^-)/2 \text{ in }\NegS.
\end{equation}

Between $u$ and $v\in H^1(B_1)$, the maximum and minimum  are denoted by
\begin{equation}
\label{EqnMaxMin}
u\vee v:=\max\{u,v\}, \text{ and }u\wedge v:=\min\{u,v\}.
\end{equation} 
It is elementary to establish the  following `cut-and-paste' lemma:
\begin{lem}
\label{LemCAndP}
For $u,v\in H^1(B_1)$, we have
$$
\En_\pm(u\vee v)+\En_\pm(u\wedge v)=\En_\pm(u)+\En_\pm(v).
$$
\end{lem} 
Consequently, we have
\begin{cor}
\label{CorCAndP}
Suppose that $u\in\mathcal{M}[\En_\pm,u]$ and $v\in\mathcal{M}[\En_\pm,v]$ with 
$
u\ge v \text{ on }\partial B_1,
$
then 
$$
u\vee v\in\mathcal{M}[\En_\pm, u]\quad \text{ and }\quad u\wedge v\in\mathcal{M}[\En_\pm, v].
$$
\end{cor}
\begin{proof}
We have that $u\vee v\in H^1_0(B_1)+u$ and $u\wedge v\in H^1_0(B_1)+v$, from the ordering $u\ge v$ on $\partial B_1$ (in the trace sense). By minimality, $\En_\pm(u\vee v) \ge \En_\pm(u)$ and $\En_\pm(u\wedge v) \ge \En_\pm(v)$, and from Lemma~\ref{LemCAndP} they are actually equalities. 
\end{proof}

We need minimizers to be continuous in $\overline{B_1}$. Although weaker assumptions could suffice, we choose to work with Lipschitz data for simplicity. We have the following:
\begin{prop}
\label{PropHolder}
For a Lipschitz function $\varphi$ on $\partial B_1$, denote by $A$ and $B$ the following quantities
$$
A:=\|\varphi\|_{L^\infty(\partial B_1)}, \text{ and }B:=\|\nabla\varphi\|_{L^\infty(\partial B_1)}.
$$
Suppose that $u\in\mathcal{M}[\En_\pm,\varphi]$, then $u\in C^{\alpha}(\overline{B_1})$ for any $\alpha\in(0,1)$ with 
$$
\|u\|_{C^\alpha(\overline{B_1})}\le C,
$$
where $C$ depends only on $d$, $\alpha$, $A+B$, and $\sup_{t\in[0,A]}F_\pm(t)$.
\end{prop} 
\begin{proof}
With \eqref{EqnSecondEL} and \eqref{EqnConditionOnFpm}, we see that $u^\pm$ are subharmonic in $B_1$, implying
$
\|u\|_{L^\infty(B_1)}\le A. 
$

For an open subset $U\subset B_1$, suppose $v\in H^1(U)$ satisfy
\begin{equation}
\label{EqnExtraConditionOnTesting}
v-u\in H^1_0(U),\hem \text{ and }-A \le v\le A \hem\text{ in }U.
\end{equation} 
The minimizing property of $u$ implies
\begin{align*}
\int_U \Diri&\le \int_U|\nabla v|^2+F_+(v^+)\chi_{\{v>0\}}+F_-(v^-)\chi_{\{v<0\}}\\&\le\int_U|\nabla v|^2+\sup_{t\in[0,A]}F_\pm(t)\cdot |U|.
\end{align*}

From here, we can apply the argument in Appendix B of Edelen-Spolaor-Velichkov \cite{EdSV}. Note in that argument the minimizing property is  tested against functions satisfying \eqref{EqnExtraConditionOnTesting}.
\end{proof} 

\subsection{Results on the one-phase Alt-Caffarelli functional}
We begin with some properties of minimizers for the one-phase Alt-Caffarelli functional $\En_{AC}^+(\cdot)$ from \eqref{EqnAC}. We consider only non-negative minimizers in this subsection. 
\begin{prop}[Alt-Caffarelli \cite{AC}]
\label{PropBasicAC}
Suppose that $u\in\mathcal{M}[\En_{AC}^+,u]$, then it satisfies
$$
\Delta u=0 \text{ in }\PosS, \text{ and }|\nabla u|=1 \text{ along }\GU,
$$
and
$$
\Delta u=\mathcal{H}^{d-1}|_{\GU},
$$
where $\mathcal{H}^{d-1}$ denotes the $(d-1)$-dimensional Hausdorff measure. 

Further assume that $x_0\in\GU\cap B_{1/2}$, then for dimensional constants $c,C>0$, we have
$$
\sup_{B_r(x_0)}u\ge cr \text{ for all $0<r<\frac14$},
$$
$$
|\nabla u|\le C \text{ in }B_{3/4}, \hem\text{ and }\mathcal{H}^{d-1}(B_{3/4}\cap\GU)\le C.
$$
\end{prop} 
Recall the set of minimizers $\mathcal{M}[\cdot,\cdot]$ defined in \eqref{EqnSetOfMin}, and the free boundary $\GU$ in \eqref{EqnFreeBoundary}. 

We have the following general lemma, about the convergence of minimizers:
\begin{lem}[\protect{\cite[Lemma 6.3]{V}}]\label{LemDiriConvAC}
Let $u_k\in \M[\En_{AC}^+,u_k]$ for $k\in \N$ be a sequence of minimizers uniformly bounded in $H^1(B_1)$. Then, up to a subsequence, $u_k$ converges strongly in $H_{\rm loc}^1(B_1)$ to some $u_\infty\in \M[\En_{AC}^+,u_\infty]$, and the sequence $\chi_{\{u_k > 0\}}$ converges to $\chi_{\{u_\infty > 0\}}$ strongly in $L^1_{\rm loc}(B_1)$ and pointwise almost everywhere in $B_1$. 
\end{lem}

For a function $u\in H^1(B_1)$, a point $x_0\in B_1$ and $r\in(0,1-|x_0|)$, the \textit{Weiss balanced energy} with center $x_0$ and at scale $r$ is defined as
\begin{equation}
\label{EqnWeissAC}
W(u,x_0,r):=\frac{1}{r^{d}}\int_{B_r(x_0)}(\Diri+\chi_{\PosS})-\frac{1}{r^{d+1}}\int_{\partial B_r(x_0)}u^2.
\end{equation} 
When there is no ambiguity, we omit $u$ or $x_0$ from the expression. 

Its usefulness lies in the following theorem:
\begin{thm}[Weiss \cite{W}]
\label{ThmWeissAC}
Suppose that $u\in\mathcal{M}[\En_{AC}^+,u]$ and $x_0\in\ConS\cap B_1$. 

For $0<s<r<1-|x_0|$, we have
$$
W(u,x_0,r)-W(u,x_0,s)=2\int_s^r\frac{1}{\rho^d}\int_{\partial B_\rho(x_0)}\left(\nabla u\cdot\nu-\frac{u}{|x|}\right)^2d\HausM d\rho.
$$
In particular, the function $r\mapsto W(r)$ is non-decreasing, and the following limit is well-defined
$$
W(u,x_0,0+):=\lim_{r\downarrow 0}W(u,x_0,r).
$$ 
If the function $r\mapsto W(r)$ is constant, then $u$ is $1$-homogeneous with respect to $x_0$.
\end{thm} 

For $u\in\mathcal{M}[u]$ and $x_0\in\GU$, define the \textit{rescaled solution} at point $x_0$ and scale $r$ as
\begin{equation}
\label{EqnRescaledSolAC}
u_{x_0,r}(x):=\frac{1}{r}u(x_0+rx).
\end{equation}  
When there is no ambiguity, we omit the point $x_0$.

Theorem \ref{ThmWeissAC} allows the following blow-up analysis: 
\begin{prop}[\cite{W}]
\label{PropBlowUpAC}
Suppose that $u\in\mathcal{M}[\En_{AC}^+, u]$ and $x_0\in\GU\cap B_1$, then along a subsequence of $r_k\downarrow 0$, the rescaled solutions $u_{x_0,r_k}$ converge to some $u_{x_0}$ locally uniformly in $\R^d$.

The limit  $u_{x_0}$ is a $1$-homogeneous minimizer\footnote{Homogeneous minimizers are also referred to as minimizing cones.} for the one-phase Alt-Caffarelli functional $\En_{AC}^+(\cdot)$ in $\R^d$.
\end{prop} 

If we blow up a homogeneous minimizer, we have the following `cone-splitting' result:
\begin{lem}[See, e.g., \protect{\cite[Lemma 10.9]{V}}]
\label{LemConeSplittingAC}
Suppose that $u$ is a minimizing cone for the Alt-Caffarelli functional $\EnAC(\cdot)$ in $\R^d$ with $e_1\in\GU$. 

If $v$ is a subsequential limit of the rescaled solutions $u_{e_1,r}$ as in Proposition \ref{PropBlowUpAC}, then 
$
\partial_1v\equiv 0
$
in $\R^d$. 
\end{lem} 
Here and in the remaining of the article, the unit vector along the $j$th coordinate axis is denoted by $e_j$ for $j=1,2,\dots,d$. Along a direction $e\in\Sph$, the directional derivative of a function $f$ is denoted by $\partial_ef$. When $e=e_j$, we simplify the notation to be $\partial_jf$.

The previous lemma is useful when combined with the following:
\begin{lem}[See, e.g., \protect{\cite[Lemma 10.10]{V}}]
\label{LemInduceInLowerDimensionAC}
Suppose that $u$ minimizes the energy $\EnAC(\cdot)$ in $\R^d$, and satisfies
$
\partial_1u\equiv 0 \text{ in }\R^d.
$

Define $\bar{u}:\R^{d-1}\to\R$ by $\bar{u}(x')=u(x',0)$, then $\bar{u}$ minimizers $\EnAC(\cdot)$ in $\R^{d-1}$.
\end{lem} 

With  Proposition \ref{PropBlowUpAC}, we  give the definition of the decomposition $\GU=\RedGU\cup\SingU$ as mentioned in the Introduction. 
\begin{defi}
\label{DefRegSingAC}
For $u\in\mathcal{M}[\EnAC, u]$ with $x_0\in\GU$, we say that $x_0$ is a \textit{regular point}, and write
$$
x_0\in\RedGU,
$$
 if $u_{x_0,r}$ converges, along a subsequence, to a rotation of $(x_d)^+$. 
 
 Otherwise, we say that $x_0$ is a \textit{singular point}, and write
 $$
 x_0\in\SingU.
 $$
\end{defi} 

Being in $\RedGU$ is an `open condition', as described in the following:
\begin{lem}[\cite{D1}]
\label{LemEpsRegAC}
Suppose that $u\in\mathcal{M}[\EnAC, u]$ satisfies 
$$
|u-(x_d)^+|<\eps \text{ in }B_1.
$$
There is a dimensional constant $\bar{\eps}>0$ such that if $\eps<\bar\eps$, then
$$
\GU\cap B_{1/2}\subset\RedGU.
$$
\end{lem} 
As a result, the singular part $\SingU$ is stable:
\begin{prop}
\label{PropStabSingAC}
For a sequence $u_k\in\mathcal{M}[\EnAC, u_k]$, suppose we have
$u_k\to u$ locally uniformly in $B_1.$

If $x_k\in\Gamma(u_k)$ for each $k$ and $x_k\to x_\infty\in B_1$, then $x_\infty\in\GU$. 

If we further assume that $x_k\in\mathrm{Sing}(u_k)$ for each $k$, then $x_\infty\in\SingU.$ 
\end{prop} 

\begin{proof}
With $u_k(x_k)=0$, locally uniform convergence gives $u(x_\infty)=0$. 
Meanwhile, for each $r>0$, we have
$$
\sup_{B_r(x_\infty)}u\ge\sup_{B_r(x_\infty)}u_k-\sup_{B_r(x_\infty)}|u-u_k|.
$$
With $x_k\to x_\infty$ and locally uniform convergence of $u_k\to u$, we have
$$
\sup_{B_r(x_\infty)}u\ge\sup_{B_{r/2}(x_k)}u_k-cr/2\ge cr/2>0,
$$
for $k$ large enough, where $c$ is the constant from Proposition \ref{PropBasicAC}. This being true for all $r>0$, we conclude
$
x_\infty\in\GU.
$

Assume now that $x_k\in\mathrm{Sing}(u_k)$ for each $k$, with $x_k \to x_\infty\in \Gamma(u)$. We want to show that $x_\infty\in\SingU$.

Up to a translation, let us assume $x_\infty = 0$, and suppose that the statement does not hold. That is, $0\in\RedGU$ and, up to a rotation, 
$$
|u-(x\cdot e_d)^+|<\frac14\bar\eps r \hem \text{ in }B_r 
$$
for some $r>0$ small, where $\bar{\eps}$ is the constant from Lemma \ref{LemEpsRegAC}.

With $x_k\to 0$ and $u_k\to u$, we have
$$
|u_k-((x-x_k)\cdot e_d)^+|<\frac12\bar\eps r \text{ in }B_{r/2}(x_k) 
$$
for large $k$. Applying Lemma \ref{LemEpsRegAC}, we conclude 
$
x_k\in\Gamma^*(u_k),
$
a contradiction.
\end{proof} 

We collect the main results on the regularity of the free boundary:
\begin{thm}[See, e.g., \protect{\cite[Theorem 1.2 and Proposition 10.13]{V}}]
\label{ThmRegOfFBAC}
For $u\in\mathcal{M}[\EnAC, u]$ in $B_1\subset\R^d$, the regular part of the free boundary $\RedGU$ is relatively open in $\GU$, and is locally a smooth hypersurface. 

For the singular part $\SingU$, we have 
\begin{enumerate}
\item{ If $d\le\Dac$, then $\SingU=\emptyset$;}
\item{ If $d=\Dac+1$, then $\SingU$ is locally discrete;}
\item{ For $d\ge\Dac+2$, then $\DimH(\SingU)\le d-\Dac-1$.}
\end{enumerate}
\end{thm} 
\noindent (Recall the dimension $\Dac$ from \eqref{EqnCriticalAC}.)

In particular, in low dimensions, the free boundary is smooth. In general dimensions, the regularity of free boundaries can be achieved by imposing monotonicity:
\begin{thm}[\cite{DJ}]
\label{ThmGraphicalFBAC}
Suppose that $u\in\mathcal{M}[\EnAC, u]$ with
$$
\partial_1u\ge 0 \text{ in }B_1.
$$
If there is a small $\delta>0$ such that 
$$
u=0 \text{ for }x_1\le-1+\delta, \text{ and }u>0 \text{ for }x_1>1-\delta,
$$
then $\GU\cap B_{\delta}=\RedGU\cap B_{\delta}$. 
\end{thm} 
\begin{rem}
By the results in  \cite{EFeY}, the same conclusion holds for viscosity solutions. 
\end{rem} 
For a point $p\in\RedGU$, we use the following notation
\begin{equation}
\label{EqnUnitNormal}
\nu_p:=\text{ the unit normal to the free boundary, pointing towards }\PosS.
\end{equation} 
\subsection{Results on the one-phase Alt-Phillips functional}
The study of the Alt-Phillips functional $\EnAP(\cdot)$ from \eqref{EqnAP} parallels the study of the Alt-Caffarelli functional. We consider only nonnegative minimizers in this subsection. 

We focus on the case when the parameter $\gamma\in(0,1)$. Corresponding to this $\gamma$, we introduce the parameter $\beta$, which dictates the natural homogeneity of the problem:
\begin{equation}
\label{EqnBeta}
\beta:=\frac{2}{2-\gamma}\in(1,2).
\end{equation} 

We begin with  basic properties of the minimizer:
\begin{prop}[Alt-Phillips \cite{AP}]
\label{PropBasicAP}
Suppose that $u\in\mathcal{M}[\EnAP, u]$,  then it satisfies
$$
\Delta u=\frac{\gamma}{2} u^{\gamma-1}\chi_{\PosS} \text{ in $B_1$.}
$$

If we further assume that $x_0\in\GU\cap B_{1/2}$, then for constants $c,C>0$ depending only on $d$ and $\gamma$, we have
$$
\sup_{B_r(x_0)}u\ge cr^\beta \text{ for all } r\in(0,1/4), 
$$ 
and 
$$\|u\|_{C^{1,\beta-1}(B_{3/4})}\le C.
$$
\end{prop} 
Recall the set of minimizers $\mathcal{M}[\cdot,\cdot]$ from \eqref{EqnSetOfMin}.

With this proposition, a uniformly bounded family of minimizers is compact in the $C^1$-topology. As a result, the Dirichlet energy is continuous with respect to the convergence of minimizers:
\begin{lem}\label{LemDiriConvAP}
Let $u_k\in \M[\En_{AP}^+,u_k]$ for $k\in \N$ be a sequence of minimizers uniformly bounded in $H^1(B_1)$. Then, up to a subsequence, $u_k$ converges strongly in $H_{\rm loc}^1(B_1)$ to some $u_\infty\in \M[\En_{AP}^+,u_\infty]$. 
\end{lem}

For the Alt-Phillips problem, the corresponding \textit{Weiss balanced energy} is defined as follows. Since there is no room for ambiguity,  we use the same notation as we used in \eqref{EqnWeissAC}.  

For $u\in\HOne$, $x_0\in B_1$ and $r\in(0,1-|x_0|)$, we define
\begin{equation}
\label{EqnWeissAP}
W(u,x_0,r):=\frac{1}{r^{d+2\beta-2}}\int_{B_r(x_0)}(\Diri+u^\gamma)-\frac{\beta}{r^{d+2\beta-1}}\int_{\partial B_r(x_0)}u^2.
\end{equation} 
We often omit $u$ or $x_0$ from the expression. 

For minimizers, this is a monotone function in $r$:
\begin{thm}[Weiss \cite{W}]
\label{ThmWeissAP}
Suppose that $u\in\mathcal{M}[\En_{AP}^+,u]$ and $x_0\in\ConS\cap B_1$. 

For $0<s<r<1-|x_0|$, we have
$$
W(u,x_0,r)-W(u,x_0,s)=2\int_s^r\frac{1}{\rho^{d+2\beta-2}}\int_{\partial B_\rho(x_0)}\left(\nabla u\cdot\nu-\beta\frac{u}{|x|}\right)^2d\HausM d\rho.
$$
In particular, the function $r\mapsto W(r)$ is non-decreasing, and the following limit is well-defined
$$
W(u,x_0,0+):=\lim_{r\to 0}W(u,x_0,r).
$$ 
If the function $r\mapsto W(r)$ is constant, then $u$ is $\beta$-homogeneous with respect to $x_0$.
\end{thm} 

This allows the blow-up analysis:
\begin{prop}[Weiss \cite{W}]
\label{PropBlowUpAP}
Suppose that $u\in\mathcal{M}[\En_{AP}^+, u]$ and $x_0\in\GU\cap B_1$, define the rescaled solutions as
\begin{equation}
\label{EqnRescaledSolAP}
u_{x_0,r}(x):=\frac{1}{r^\beta}u(x_0+rx).
\end{equation}

Then along a subsequence of $r_k\downarrow 0$, the rescaled solutions $u_{x_0,r_k}$ converge to some $u_{x_0}$ locally uniformly in $\R^d$.

The limiting function $u_{x_0}$ is a $\beta$-homogeneous minimizer for  the Alt-Phillips functional $\En_{AP}^+(\cdot)$ in $\R^d$.
\end{prop} 

Although the values of $W(0+)$ in Theorem \ref{ThmWeissAP} have not been classified, for minimizing cones, we have the following interpretation:
\begin{lem}
\label{LemWeissRepAP}
Suppose that $u\in\mathcal{M}[\En_{AP}^+, u]$  is $\beta$-homogeneous, then 
$$
W(u,0,0+)=W(u,0,1)=\left(1-\frac{\gamma}{2}\right)\int_{B_1}u^\gamma.
$$
\end{lem} 
\begin{proof}
The first equation follows from the homogeneity of $u$ and the scaling symmetry of $W(\cdot)$. 

To see the second, we use the homogeneity of $u$ and the equation in Proposition~\ref{PropBasicAP} to see
$$
\int_{B_1}\Diri=\int_{\partial B_1}uu_\nu-\int_{B_1}u\Delta u=\beta\int_{\partial B_1}u^2-\int_{B_1}\frac{\gamma}{2}u^\gamma
$$
We use the definition of the Weiss balanced energy to get
$$
W(1)=\int_{B_1}(\Diri+u^\gamma)-\beta\int_{\partial B_1}u^2=\left(1-\frac{\gamma}{2}\right)\int_{B_1}u^\gamma,
$$
as we wanted.
\end{proof} 

When blowed up at a non-zero free boundary point, minimizing cones split:
\begin{lem}
\label{LemConeSplittingAP}
Suppose that $u$ is a minimizing cone for the Alt-Phillips functional $\EnAP(\cdot)$ with $e_1\in\GU$. 

If $v$ is a subsequential limit of the rescaled solutions $u_{e_1,r}$ as in Proposition \ref{PropBlowUpAP}, then
$
\partial_1v\equiv 0
$
in $\R^d$.
\end{lem} 
The proof follows by modifying the argument in \cite[Lemma 10.9]{{V}}.

For a minimizer $u$, its free boundary decomposes into the regular and singular parts.
\begin{defi}
\label{DefRegSingAP}
For $u\in\mathcal{M}[\EnAP, u]$ with $x_0\in\GU$, we say that $x_0$ is a \textit{regular point}, and write
$$
x_0\in\RedGU,
$$
 if $u_{x_0,r}$ converges, along a subsequence, to a rotation of $[(x_d^+/\beta)]^\beta$. 
 
 Otherwise, we say that $x_0$ is a \textit{singular point}, and write
 $$
 x_0\in\SingU.
 $$
\end{defi} 

Being in the regular part $\RedGU$ is an open condition:
\begin{lem}
\label{LemEpsRegAP}
Suppose that $u\in\mathcal{M}[\EnAP,u]$  satisfies
\begin{equation}
\label{EqnFlatnessConditionAP}
|u-[(x_d^+/\beta)]^\beta|<\eps \text{ in }B_1
\end{equation}
for some $\eps>0$.

There are constants $\bar{\eps},\bar{r}>0$, depending only on $d$ and $\gamma$,  such that if $\eps<\bar\eps$, then 
$$
\GU\cap B_{\bar{r}}\subset\RedGU.
$$
\end{lem} 

\begin{proof}
With Theorem 6.1 in Alt-Phillips \cite{AP}, it suffices to show that our assumption \eqref{EqnFlatnessConditionAP}, for $\eps$ small enough,  implies that the minimizer $u$ is `flat' as in Definition 5.1 of Alt-Phillips. 

Under \eqref{EqnFlatnessConditionAP}, it remains to show that $u=0$ in $\{x_d\le -\sigma\}\cap B_{1/2}$ with $\sigma\to0$ as $\eps\to0.$ This follows from the lower bound in Proposition \ref{PropBasicAP}, which implies that 
$$
u=0 \text{ in }B_{1/2}\cap\{x_d<-C\eps^{1/\beta}\}
$$
for a dimensional constant $C>0.$
\end{proof} 

With similar argument as in the proof for Lemma \ref{PropStabSingAC}, this leads to the stability of the singular part $\SingU$:
\begin{prop}
\label{PropStabSingAP}
Suppose that for a sequence $u_k\in\mathcal{M}[\EnAP, u_k]$, we have
$
u_k\to u\in\mathcal{M}[\EnAP, u] \text{ locally uniformly in }B_1.
$

If $x_k\in\Gamma(u_k)$ for each $k$ and $x_k\to x_\infty\in B_1$, then $x_\infty\in\GU$. 

If we further assume that $x_k\in\mathrm{Sing}(u_k)$, then $x_\infty\in\SingU.$ 
\end{prop}

The regular part $\RedGU$ is a regular surface.
\begin{thm}[Alt-Phillips \cite{AP}]
\label{ThmRegSmoothAP}
For $u\in\mathcal{M}[\EnAP, u]$, the regular part $\RedGU$ is relatively open in the free boundary $\GU$, and is locally a $C^{1,\alpha}$-hypersurface. 
\end{thm} 

Concerning the singular part $\SingU$, its Hausdorff dimension is expected to enjoy similar bounds as in Theorem \ref{ThmRegOfFBAC}, as we stated in Theorem \ref{ThmSingAP}. Unfortunately, to the best knowledge of the authors, such a result has not been rigorously established yet.  We postpone its proof to Appendix \ref{AppSingAP}.

\subsection{Tools from   geometric measure theory}
We begin with an elementary lemma about homogeneous functions:
%
%

\begin{lem}
\label{LemHomogeneousFunctions}
Let $f$ and $g$ be continuous  and $m$-homogeneous functions. Then:
$$
f(\,\cdot\, +e_1)\ge f \quad\text{in}\quad \R^d\quad \Longrightarrow  \quad  f(\,\cdot\,+se_1)\ge f\quad\text{in}\quad \R^d, \ \text{for all}\ s > 0,
$$
and
 $$f(x+e_1)\ge g(x)\ \forall x\in\R^{d}\ \Longrightarrow \ f\ge g \text{ in }\R^d.$$ 
\end{lem} 

\begin{proof}
For the first statement, we notice that the homogeneity of $f$  implies
$$
f(x+se_1)=s^mf(x/s+e_1)\ge s^mf(x/s)=f(x) \text{ for all }s>0.
$$

For the second statement, we notice that for all $s>0$, we have
$$
f(x+se_1)=s^mf(x/s+e_1)\ge s^mg(x/s)=g(x) \text{ in }\R^d.
$$
Sending $s\to0$, the result follows from the continuity of $f$. 
\end{proof} 

Given a set with positive Hausdorff measure,  we can always find a point of positive density:
\begin{lem}[See,e.g., \protect{\cite[Lemma 10.5]{V}}]
\label{LemExistenceOfDensityPoints}
For $s>0$, let $\mathcal{H}^{s}(\cdot)$ denotes the $s$-dimensional Hausdorff measure. 

For a set $E$ in $\R^d$, if $\mathcal{H}^s(E)>0$, then we can find a point $x_0\in E$ such that 
$$
\limsup_{r\to0}\frac{\mathcal{H}^s(E\cap B_r(x_0))}{r^s}>0.
$$
\end{lem} 

Given a subset $E$ in $\R^d$ and a real-valued function $f$ on $E$, the following lemma allows us to consider only points of continuity of $f$:
\begin{lem}[\protect{\cite[Lemma 7.1]{FRS}}]
\label{LemPointsOfContinuity}
For $E\subset\R^d$ and $f:E\to\R$, define
$$
F:=\{x\in E:\hem \text{there exists a sequence }x_k\in E \text{ such that }x_k\to x  \text{ and } f(x_k)\to f(x)\}.
$$
Then $E\backslash F$ is countable. 
\end{lem} 

The following lemma allows us to estimate the dimension of a set, considering only points of continuity of a function:
\begin{lem}[\protect{\cite[Proposition 7.3]{FRS}}]
\label{LemReifenberg}
Suppose that $E\subset\R^d$ and $f:E\to\R$ satisfy the following:

For each $\eps>0$ and $x\in E$, there exists $\rho=\rho(x,\eps)>0$ such that for each $r\in(0,\rho)$, we can find an $m$-dimensional subspace $\Pi_{x,r}$, passing through $x$ and satisfying
$$
E\cap B_r(x)\cap \{y:\hem f(x)-\rho<f(y)<f(x)+\rho\}\subset\{y:\hem \mathrm{dist}(y,\Pi_{x,r})\le \eps r\}.
$$

Then the Hausdorff dimension of $E$ can be  bounded from above as 
$$
\DimH(E)\le m.
$$
\end{lem} 

In this work, for a set $E$ in space-time $\R^d\times\R$, we use $\pi_x:E\to\R^d$ and $\pi_t:E\to\R$ to denote the canonical projections, that is, 
\begin{equation}
\label{EqnSpaceTimeProjections}
\pi_x(x,t)=x, \text{ and }\pi_t(x,t)=t.
\end{equation} 

The following is fundamental  for Hausdorff measure-based generic regularity from \cite{FRS}. We state a special case  for our purpose:
\begin{lem}[\protect{\cite[Corollary 7.8]{FRS}}]
\label{LemGenericReduction}
Suppose that for some $s>0$ and $\gamma\in[s,d]$, the set $E\subset\R^d\times(-1,1)$ satisfies
\begin{enumerate}
\item{$\DimH(\pi_x(E))\le\gamma$; and }
\item{For each $(x_0,t_0)\in E$ and $\eps>0$,  there exists $\rho=\rho(x_0,t_0,\eps)>0$ such that 
$$
E\cap\{(x,t)\in B_{\rho}(x_0)\times(-1,1):\hem t-t_0>|x-x_0|^{s-\eps}\}=\emptyset.
$$}
Then we have
$$
\DimH(E\cap\pi_t^{-1}(t))\le\gamma-s\hem \text{ for almost every } t\in(-1,1).
$$
\end{enumerate}
\end{lem}

\section{Generic uniqueness for general energy functionals}
\label{sec:3}
In this section, we study the generic uniqueness of minimizers for the energy  $\En_\pm(\cdot)$ in \eqref{EqnFreeBoundaryEnergy} and prove Theorem \ref{ThmGU}.  For the nonlinearities $F_\pm$, we assume \eqref{EqnConditionOnFpm}. Recall  the set of minimizers $\mathcal{M}[\cdot,\cdot]$ defined in \eqref{EqnSetOfMin}.

\vem

We begin with an elementary lemma concerning the topology of super level sets of a minimizer:
\begin{lem}
\label{LemComponentsGoesToBoundary}
Suppose that $u\in\mathcal{M}[\varphi]$ for a Lipschitz function $\varphi$. 

Let $\mathcal{C}$ be a connected component of $\PosS\cap B_1$.  Then,  there is a point $x_0$ such that 
$$
x_0\in \overline{\mathcal{C}}\cap\{\varphi>0\}\cap\partial B_1. 
$$
Moreover, if $\mathcal{D}$ is the  connected component of $\{\varphi>0\}\cap\partial B_1$  containing $x_0$, then 
$$
\overline{\mathcal{C}}\supset\mathcal{D}.
$$

Analogous results hold for connected components of $\NegS\cap B_1.$
\end{lem} 
%
%
%
%
%

 \begin{proof}
Proposition \ref{PropHolder} implies that $u\in C(\overline{B_1})$.

Suppose, by contradiction that 
$
\overline{\mathcal{C}}\cap\{\varphi>0\}\cap\partial B_1=\emptyset.
$ Then we have
$
\partial\mathcal{C}\subset\{u\le0\}, 
$
and 
$
\Delta u\ge 0 \text{ in }\mathcal{C}
$
(by \eqref{EqnSecondEL}). Maximum principle implies that $u\le 0$ in $\mathcal{C}$, contradicting the definition of $\mathcal{C}$.

Finally, by continuity of $u$, we immediately have $\mathcal{D}\subset \overline{\mathcal{C}}$.
\end{proof}

With this, we prove the ordering property for minimizers:
\begin{prop}
\label{PropComparison}
For a family of boundary data $\FamPhi$ satisfying Assumption \ref{AssGU}, suppose that 
$$
u_t\in\mathcal{M}[\varphi_t], \text{ and }\hem u_s\in\mathcal{M}[\varphi_s] \hem\text{ for some }t>s.
$$
Then 
$$
u_t\ge u_s \text{ in }B_1.
$$
\end{prop} 

\begin{proof}
With our notation \eqref{EqnMaxMin}, we take
$$
V:=u_t\vee u_s \text{ and }v:=u_t\wedge u_s.
$$
Since   $\varphi_t\ge\varphi_s$ on $\partial B_1$, 
 Corollary \ref{CorCAndP} implies
$$
V\in\mathcal{M}[\varphi_t] \text{ and }v\in\mathcal{M}[\varphi_s].
$$

\vem

\noindent \textit{Case 1: Ordering of $u_t$ and $u_s$ in $\{u_s>0\}\cap B_1$.}

Suppose the ordering fails at a point $x_0\in\{u_s>0\}\cap B_1$, then 
$
u_t(x_0)<u_s(x_0),
$
which implies
\begin{equation}
\label{EqnGUTempo1}
V(x_0)=u_s(x_0).
\end{equation}

Let $\mathcal{C}$ denote the connected component of $\{u_s>0\}\cap B_1$ which contains $x_0$. Then we have
$$
V\ge u_s>0 \text{ in }\mathcal{C}.
$$
With $V\in\mathcal{M}[\varphi_t] $ and $u_s\in\mathcal{M}[\varphi_s]$, by \eqref{EqnSecondEL}, we have
$$
\Delta V=F'_+(V)/2, \text{ and }\Delta u_s=F'_+(u_s)/2 \text{ in }\mathcal{C}.
$$
The concavity of $F_+$ as in \eqref{EqnConditionOnFpm} leads to
$$
\Delta (V-u_s)\le 0 \text{ in }\mathcal{C}.
$$
Thanks to  \eqref{EqnGUTempo1}, the strong maximum principle forces
$
V\equiv u_s \text{ in }\mathcal{C}.
$
In particular, $V = u_s$ on a connected component of $\{u_s > 0\}\cap \partial B_1$ by Lemma~\ref{LemComponentsGoesToBoundary}, but since $V = u_t$ on $\partial B_1$, we get a contradiction with Assumption~\ref{AssGU}-\eqref{it:333}.

As a result, we have
$$
u_t\ge u_s \text{ in }\{u_s>0\}\cap B_1.
$$

In particular, if there is a point where the ordering fails, it has to lie in $\{u_t<0\}\cap\{u_s\le 0\}\cap B_1$. Consequently, it suffices to show establish the following:

\vem

\noindent \textit{Case 2: Ordering of $u_t$ and $u_s$ in $\{u_t<0\}\cap B_1.$} It follows like the previous case: if there is a point $x_0\in\{u_t<0\}\cap B_1$ such that $u_t(x_0)<u_s(x_0)$, then we have
$
v(x_0)=u_t(x_0).
$

In $\mathcal{C}$ a connected component of $\{u_t < 0 \}\cap B_1$ containing $x_0$, we have (by \eqref{EqnSecondEL}), 
$$
\Delta v=-F'_-(v^-)/2, \text{ and }\Delta u_t=-F'_-(u_t^-)/2 \text{ in }\mathcal{C},
$$
and the concavity of $F_-$ implies
$$
\Delta (v-u_t)\ge0 \text{ in }\mathcal{C}. 
$$
From $v(x_0)=u_t(x_0)$, the same argument as in  \textit{Step 1} leads to a contradiction  with Assumption~\ref{AssGU}-\eqref{it:333}.
\end{proof} 

Before the proof of Theorem \ref{ThmGU}, we introduce some notations. For a function $f:B_1\to\R$, we denote its epigragh and hypograph by $\Epi(f)$ and $\Hyp(f)$ respectively, that is,
$$
\Epi(f):=\{(x,y)\in B_1\times\R: y>f(x)\}, \text{ and }\Hyp(f):=\{(x,y)\in B_1\times\R: y<f(x)\}.
$$

\begin{proof}[Proof of Theorem \ref{ThmGU}]
Define $I\subset(-1,1)$ as 
$$
I:=\{t\in(-1,1):\hem\mathcal{M}[\varphi_t] \text{ contains at least 2 elements}\},
$$
we need to show that $I$ is at most countable.

Suppose that $t,s\in I$ with $t>s$. By the definition of $I$, we  find $u_t,v_t\in\mathcal{M}[\varphi_t]$, and a point $x_0$ such that $u_t(x_0)<v_t(x_0)$. By the continuity of $u_t$ and $v_t$,  we find a nontrivial ball $\mathcal{B}(t)\subset\Epi(u_t)\cap\Hyp(v_t)$. Similarly, we find
$u_s,v_s\in\mathcal{M}[\varphi_s]$ and  a nontrivial ball $\mathcal{B}(s)\subset\Epi(u_s)\cap\Hyp(v_s)$.

Proposition \ref{PropComparison} implies that $u_t\ge v_s$ in $B_1$. As a result, we have
$
\Epi(u_t)\cap\Hyp(v_s)=\emptyset,
$
which implies 
$$
\mathcal{B}(t)\cap\mathcal{B}(s)=\emptyset.
$$

Therefore, the collection $\{\mathcal{B}(t)\}_{t\in I}$ consists of nontrivial balls that are mutually disjoint. Such a collection is at most countable.
\end{proof} 

As an immediate consequence we have:
\begin{proof}[Proof of Corollary~\ref{CorGU}]
Follows directly from Theorem~\ref{ThmGU} by regularizing the boundary datum in the corresponding space and perturbing it by an arbitrary small constant. 
\end{proof}

\section{Generic regularity  in the Alt-Caffarelli problem}
\label{SecGRAC}
Starting from this section, we consider nonnegative minimizers for the one-phase Alt-Caffarelli energy \eqref{EqnAC} and the one-phase Alt-Phillips energy  \eqref{EqnAP}. The goal is to establish the generic regularity of  free boundariesfrom Theorems~\ref{ThmGRAC} and Theorem~\ref{ThmGRAP}.

\subsection{General discussions}
Let $\En(\cdot)$ denote either the Alt-Caffarelli or the Alt-Phillips functional, and let $\FamPhi$ be a family of boundary data satisfying Assumption~\ref{AssGR}. We are interested in the collection of free boundary points in space-time, that is,
\begin{equation}
\label{EqnSpaceTimeFreeBoundary}
\mathcal{G}:=\{(x,t)\in B_1\times(-1,1):\hem x\in\Gamma(u_t) \hem\text{ for some }\hem u_t\in\mathcal{M}[\En,\varphi_t]\},
\end{equation} 
where $\mathcal{M}[\cdot,\cdot]$ is the set of minimizers (see \eqref{EqnSetOfMin}), and $\Gamma(\cdot)$ is the free boundary (see~\eqref{EqnFreeBoundary}).

We are concerned with the set of singularities, namely,
\begin{equation}
\label{EqnSpaceTimeSing}
\mathcal{S}:=\{(x,t)\in B_1\times(-1,1):\hem x\in\mathrm{Sing}(u_t) \hem\text{ for some }\hem u_t\in\mathcal{M}[\En,\varphi_t]\},
\end{equation} 
where the singular part $\mathrm{Sing}(\cdot)$ is defined in Definitions~\ref{DefRegSingAC} and~\ref{DefRegSingAP}.
Thanks to Theorem \ref{ThmGU}, we can discard a countable subset of $(-1,1)$,  and consider the reduced set of singularities
\begin{equation*}
\label{EqnReducedSing}
\begin{split}\mathcal{S}''& :=\{(x,t)\in B_1\times[(-1,1)\backslash I]:\hem x\in\mathrm{Sing}(u_t) \text{ for }u_t\in\mathcal{M}[\En,\varphi_t]\}\\
& = \{(x,t)\in B_1\times(-1,1):\hem x\in\mathrm{Sing}(u_t) \text{ and }\mathcal{M}[\En,\varphi_t] = \{u_t\}\},
\end{split}
\end{equation*} 
 and by Remark~\ref{RemContAlmEverywhere} we can further consider 
\[
\begin{split}\mathcal{S}'& :=\{(x,t)\in \mathcal{S}'' : t\in (-1, 1)\setminus I_D\}.
\end{split}
\]

The set $I$, as in Theorem \ref{ThmGU}, contains the collection of time instances $t$ when $\mathcal{M}[\En, \varphi_t]$ contains more than one element, and the set $I_D$ contains those times $t$ where the maps $t\mapsto\varphi_t$ is not continuous. Both sets are countable.  

Each of these sets is stratified according to the \textit{density of the Weiss balanced energy}. To be precise, let $W(\cdot)$ denote the Weiss balanced energy from \eqref{EqnWeissAC} and \eqref{EqnWeissAP}. By  either Theorem \ref{ThmWeissAC} or Theorem \ref{ThmWeissAP}, we see that the following quantity is well-defined on $\mathcal{G}$
\begin{equation}
\label{EqnWeissLimit}
\omega(x,t):=W(u_t,x,0+).
\end{equation} 
Lemma \ref{LemPointsOfContinuity} allows us to discard a countable set (of values $t$) and  further reduce $\mathcal{S'}$ into
\begin{equation}
\label{EqnFurtherReducedSing}
\mathcal{S}^*:=\{(x,t)\in \mathcal{S}': \exists (x_k,t_k)\in\mathcal{S}' \text{ with }(x_k,t_k)\to(x,t) \text{ and }\omega(x_k,t_k)\to\omega(x,t)\}.
\end{equation} 

\vem

To establish Theorem \ref{ThmGRAC} and Theorem \ref{ThmGRAP}, we need to estimate the size of  the following set
\begin{equation}
\label{EqnSingTimes}
J:=\pi_t(\mathcal{S}),
\end{equation} 
where we use the projection $\pi_t(\cdot)$ defined in \eqref{EqnSpaceTimeProjections}. 

With the previous discussion, we decompose $J$ as 
 \begin{align}
\label{EqnDecompositionOfSingTimes}
\begin{split}
J&=I\cup I_D\cup \pi_t(\mathcal{S}')=I\cup I_D \cup\pi_t(\mathcal{S}'\backslash\mathcal{S}^*)\cup\pi_t(\mathcal{S}^*).
\end{split}
\end{align} 
By Theorem \ref{ThmGU},  Remark~\ref{RemContAlmEverywhere},  and Lemma \ref{LemPointsOfContinuity}, we know that the set $I\cup\pi_t(\mathcal{S}'\backslash\mathcal{S}^*)$ is countable.   Consequently, for the size of $J$,  it suffices to estimate the size of $\pi_t(\mathcal{S}^*)$.

In this section, we focus on this task for the Alt-Caffarelli functional and establish Theorem \ref{ThmGRAC}.

\subsection{Preparatory Lemmas}
We begin with some  properties of minimizers for the Alt-Caffarelli functional $\EnAC(\cdot)$ from \eqref{EqnAC}. Recall that we consider nonnegative minimizers.

The first one says that ordered minimizing cones coincide. 

\begin{lem}
\label{LemOrderedConesAC}
Suppose that $u$ and $v$ are minimizing cones for $\EnAC(\cdot)$ in $\R^d$ with 
$$
u\ge v \text{ in }\R^d,
$$
then
$$
u\equiv v \text{ in }\R^d.
$$
\end{lem} 
\begin{proof}
As an abuse of notation, we denote by $u$ and $v$ the restriction of the minimizers to $\Sph$. 

Since $u$ and $v$ are $1$-homogeneous, Proposition \ref{PropBasicAC} implies that 
$$
\Delta_{\Sph}u=-(d-1) u\quad \text{in}\quad \Sph\cap\PosS, \text{ and }\Delta_{\Sph}v=-(d-1)v \quad \text{in}\quad \Sph\cap\{v>0\},
$$
where $\Delta_{\Sph}$ denotes the spherical Laplacian.
The ordering $u\ge v$ implies $\PosS\supset\{v>0\}$. Meanwhile, these two sets have the same principal eigenvalue $(d-1)$, and are connected (see, e.g., \cite[Theorem 2.3]{EdSV}). The monotonicity of the principal eigenvalue with respect to domain inclusion forces
$$
\PosS=\{v>0\}.
$$
With the principal eigenspace being one-dimensional and the connectedness of the positivity sets, we have
$$
u=\alpha v \text{ for some }\alpha>0.
$$
Proposition \ref{PropBasicAC} implies that $|\nabla u|=|\nabla v|=1$ on the free boundary, thus we have $\alpha=1$.
\end{proof} 

The second property says that monotone minimizing cones are either translation invariant or half-space solutions. 

\begin{lem}
\label{LemMonotoneConesAC}
Suppose that $u$ is a minimizing cone for $\EnAC(\cdot)$ in $\R^d$ with 
$$
u(x+e_1)\ge u(x) \hem\text{ for all }x\in\R^d,
$$
then either 
$$
\partial_1u\equiv 0\hem \text{ in }\R^d,
$$
or
$$
u(x)=(x\cdot e)^+
$$
for some $e\in \mathbb{S}^{d-1}$ with $e\cdot e_1 \ge 0$. 
\end{lem} 
\begin{proof}
Lemma \ref{LemHomogeneousFunctions} implies that $u$ is monotone non-decreasing in the $e_1$-direction. 

With $0\in\GU$, this implies that $e_1\in\overline{\PosS}$ and $u(-e_1)=0$.  We have three cases to consider:
\begin{enumerate}
\item{$e_1\in\GU$;}
\item{$-e_1\in\GU$; and }
\item{There is $\delta>0$ such that 
$$
B_\delta(e_1)\subset\PosS, \text{ and }B_\delta(-e_1)\subset\ConS.
$$}
\end{enumerate}

\vem

In the first case, we take the rescaled functions
$
v_r(x):=\frac{1}{r}u(e_1+rx).
$
By Proposition \ref{PropBlowUpAC} and Lemma \ref{LemConeSplittingAC}, for some sequence $r_k\to 0$, we have
$$
v_{r_k}\to v_0 \text{ locally uniformly in }\R^d,
$$
where $v_0$ is a minimizing cone invariant along the $e_1$-direction. 

Meanwhile, with the definition of $v_r$, we have
$
v_r(x)\ge\frac{1}{r}u(rx)=u(x).
$
As a result, we have $v_0\ge u$. Lemma \ref{LemOrderedConesAC} implies that $u=v_0$. Consequently, the minimizing cone $u$ is itself invariant along the $e_1$-direction. Similar arguments work for the second case.

In the third case, the homogeneity and monotonicity of $u$ imply that in a neighborhood of $0$, the free boundary is a  bounded graph in the $e_1$-direction. Theorem~\ref{ThmGraphicalFBAC} implies that $\GU$ is smooth. Since it is a cone, the free boundary  $\GU$ is a hyperplane. Consequently $u=x_1^+$ up to a rotation. 
\end{proof} 

The following lemma states that the free boundary retracts with linear speed.
\begin{lem}
\label{LemRetractionAC}
For a family $\FamPhi$ satisfying Assumption \ref{AssGR}, suppose that, for each $t\in(-1,1)$, $u_t\in\M[\EnAC, \varphi_t]$, and let $\tau \in (-1, 1)$ fixed. Suppose, also, that $x_0\in B_{1/2}\cap\Gamma(u_{t_0})$ for some $t_0\in[\tau,1)$. 

Then, there is a constant $\kappa>0$, depending only on $\tau$, $d$, as well as $L$ and $\theta$ from Assumption \ref{AssGR}, such that 
$$
B_{\kappa\theta(t-t_0)}(x_0)\subset\{u_t>0\}
$$
for all $t\in [t_0, 1)$.
\end{lem} 
\begin{proof}
The main part of this proof is devoted to  the following claim:

\textit{Claim:} There is a constant $\kappa=\kappa(\tau,d,L,\theta)>0$, such that 
$$
\sup_{B_{\kappa\theta(t-t_0)}(x)}u_{t_0}\le u_t(x) \hem\text{ for }x\in B_{3/4} \text{ and } t\in [t_0, 1).
$$

Now, since $x_0\in \Gamma(u_{t_0})$, for $x\in B_{\kappa\theta(t-t_0)}(x_0)$, we find $x'$ such that 
$$
u_{t_0}(x')>0, \text{ and }|x'-x|<\kappa\theta(t-t_0).
$$
The previous \textit{Claim} then gives
$$
u_t(x)\ge u_{t_0}(x')>0,
$$
 the desired conclusion. 

Let us show the \textit{Claim}.

\noindent \textit{Step 1: The separation between $u_t$ and $u_{t_0}$ near $\partial B_1$.}

Under Assumption \ref{AssGR}, we have (recall $\tau \in (-1, 1)$ so $\tau + 1 > 0$)
$$
\varphi_{t_0}\ge \theta(\tau+1)>0 \text{ on $\partial B_1$.}
$$  
Proposition \ref{PropHolder} implies that 
$$
\|u_{t_0}\|_{C^{1/2}(\overline{B_1})}\le C(d,L).
$$
As a consequence, there is $\delta=\delta(\tau,d,\theta, L)<1/16$ such that 
$$
u_{t_0}>0 \text{ in }U:=B_1\backslash\overline{B_{1-8\delta}}.
$$

With Proposition \ref{PropComparison}, we have $u_t\ge u_{t_0}$ in $B_1$ for $t \ge t_0$. Proposition \ref{PropBasicAC} implies
$$
\Delta(u_t-u_{t_0})=0 \text{ in }U.
$$
Let $h$ be the solution to the following
$$
\Delta h=0 \text{ in }U, \hem h=0 \text{ on }\partial B_{1-8\delta}, \text{ and }h=1 \text{ on }\partial B_1.
$$
With Assumption \ref{AssGR}, the comparison principle implies
$$
u_t-u_{t_0}\ge\theta(t-t_0)h \text{ in }U.
$$
This leads to 
\begin{equation}
\label{EqnComparisonNearFixedBoundaryAC}
u_t-u_{t_0}\ge\kappa_0 (t-t_0) \text{ in }B_1\setminus B_{1-4\delta}
\end{equation} 
for some $\kappa_0=\kappa_0(\tau,d,\theta,L).$

\vem

\textit{Step 2: The ordering between $u_t$ and translations of $u_{t_0}$ near $\partial B_1$.}

Since both $u_t$ and $u_{t_0}$ are Lipschitz in $B_{1-\delta}$ (with constants that depend only on $\delta$), equation \eqref{EqnComparisonNearFixedBoundaryAC} implies 
\[
 \begin{split}
u_t(x) - u_{t_0}(y)& \ge u_t(x) - u_{t_0}(x)-|u_{t_0}(x)- u_{t_0}(y)| \\
& \ge \frac{\kappa_0}{2} (t-t_0)\qquad\text{for all}\quad x\in B_{1-2\delta}\setminus B_{1-3\delta},\quad y\in B_{\eta(t-t_0)}(x),
\end{split}
\]
  for some small fixed $\eta = \eta(\tau, d, \theta, L)\le\delta/2$ (that depends on the Lipschitz constant of $u$ in $B_{1-\delta}$). 

In particular, we have 
$$
u_{t_0}(x-se)< u_t(x) \hem\text{ for }e\in \mathbb{S}^{d-1}, \ x\in \partial B_{1-2\delta},\ \text{ and } 0<s<\eta(t-t_0).
$$
Applying Proposition \ref{PropComparison} to $u_{t_0}(\cdot-s e)$ and $u_{t}$ in $B_{1-2\delta}$, we have
$$
u_{t_0}(\cdot-se)\le u_t \text{ in }B_{1-2\delta}\supset B_{7/8},
$$
for all $0<s<\eta(t-t_0)$, which establishes the \textit{Claim}.
\end{proof} 

As a corollary, free boundaries do not overlap:
\begin{cor}
\label{CorNonOverLappingFBAC}
Suppose the family $\FamPhi$ satisfies Assumption \ref{AssGR} and $u_t\in\M[\EnAC,\varphi_t]$ for each $t\in(-1,1)$. Then
$$
\Gamma(u_t)\cap\Gamma(u_s)=\emptyset \text{ if }t\neq s.
$$
\end{cor}

We conclude this  subsection with a lemma on blow-ups along variable centers in space-time:
\begin{lem}
\label{LemVariableCenterAC}
For a family $\FamPhi$ satisfying Assumption \ref{AssGR}, suppose that $u_t\in\M[\EnAC, \varphi_t]$ for each $t\in(-1,1)$.

For a sequence $(x_k,t_k)\in\mathcal{G}$ satisfying 
$$
x_k\neq0, \hem (x_k,t_k)\to (0,0)\in\mathcal{G} \text{ and }\omega(x_k,t_k)\to\omega(0,0),
$$
define $r_k=|x_k|$ and 
$$
u_k(x):=\frac{1}{r_k}u_{t_k}(x_k+r_kx).
$$

If  $\mathcal{M}[\EnAC, \varphi_0]$ is a singleton and $0\notin I_D$ (see Remark~\ref{RemContAlmEverywhere}), then, up to a subsequence, we have
$$
u_k\to u_\infty \text{ locally uniformly in }\R^d,
$$
where $u_\infty$ is a minimizing cone. 
\end{lem} 
Recall the space-time free boundary  $\mathcal{G}$ in \eqref{EqnSpaceTimeFreeBoundary}, and the Weiss energy density $\omega(\cdot,\cdot)$ from \eqref{EqnWeissLimit}.

\begin{proof}
For each $k$, we have $u_k(0)=0$. Proposition \ref{PropBasicAC} implies that the family $\{u_k\}$ is locally uniformly  bounded and uniformly Lipschitz. As a result, up to a subsequence, the sequence $u_k$ converges to some minimizer $u_\infty$. 

It remains to show that $u_\infty$ is homogeneous. 

\vem
For $r>0$,  Lemma \ref{LemDiriConvAC} gives
\begin{align}
\label{EqnUpperBoundForAWeiss}
\begin{split}
W(u_\infty,0,r)&=\lim_{k\to \infty} W(u_k,0,r)\\
&=\lim_{k\to \infty} W(u_{t_k},x_k,rr_k)\\
&\ge\lim_{k\to \infty} \omega(x_k,t_k)\\
&=\omega(0,0).
\end{split}
\end{align}
The inequality follows from  the monotonicity of $r\mapsto W(r)$  as in Theorem \ref{ThmWeissAC}.

\vem

By Proposition \ref{PropHolder} and the  Lipschitz bound on $\FamPhi$ in Assumption~\ref{AssGR}, we see that the family $\{u_{t_k}\}_{k\in \N}$ is uniformly H\"older in $\overline{B_1}$. Up to a subsequence, the sequence $u_{t_k}$ converges uniformly in $\overline{B_1}$ to $v$, some minimizer of $\EnAC(\cdot)$. 

Thanks to  the continuity of $t\mapsto  \FamPhi$ at $t = 0$ ($t\notin I_D$), we see that $v\in\mathcal{M}[\varphi_0]$. Since $\mathcal{M}[\varphi_0]$ is a singleton, we have $v=u_0$. Together with $x_k\to0$, we have
$$
u_{t_k}(\cdot+x_k)\to u_0 \text{ locally uniformly in }B_1.
$$

For $\delta>0$, by the definition of $\omega(0,0)$, we find $r_\delta>0$ such that 
$$
W(u_0,0,r_\delta)<\omega(0,0)+\delta.
$$
Lemma \ref{LemDiriConvAC} imlies
$$
W(u_{t_k},x_k,r_\delta)=W(u_{t_k}(\cdot+x_k),0,r_\delta)<\omega(0,0)+2\delta
$$
for all large $k$. 

For $r>0$, this implies, for $k$ large enough,
\begin{align*}
W(u_\infty,0,r)&=\lim_{k\to \infty} W(u_{t_k},x_k,rr_k)\\
&\le\lim_{k\to \infty} W(u_{t_k},x_k,r_\delta)\\
&\le\omega(0,0)+2\delta.
\end{align*}
This being true for all $\delta>0$, we have
$$
W(u_\infty,0,r)\le\omega(0,0) \text{ for all }r>0.
$$

Combined with \eqref{EqnUpperBoundForAWeiss}, we see that $r\mapsto W(u_\infty,0,r)$ is constant. Theorem \ref{ThmWeissAC} implies that $u_\infty$ is homogeneous. 
\end{proof}

\subsection{Proof of Theorem \ref{ThmGRAC} for $d=\Dac+1$}
With these preparations, we give the proof of Theorem \ref{ThmGRAC}. In this subsection, we deal with the critical dimension $d=\Dac+1$, where $\Dac$ is  from \eqref{EqnCriticalAC}.

With the decomposition in \eqref{EqnDecompositionOfSingTimes}, it suffices to establish the following
\begin{prop}
\label{PropEmptySingInCriticalDimAC}
For $d=\Dac+1$ and a family $\FamPhi$ satisfying Assumption~\ref{AssGR}, we have
$$
\mathcal{S}^*=\emptyset,
$$
where $\mathcal{S}^*$ is the reduced set of space-time singularities from \eqref{EqnFurtherReducedSing}.
\end{prop} 

\begin{proof}
Suppose not, then without loss of generality, we assume $(0,0)\in\RedSing.$ That is, 
$$
0\in\Sing(u_0) \text{ for }u_0\in\mathcal{M}[\varphi_0],
$$
and there is a sequence $(x_k,t_k)\to(0,0)$ such that  
$$
x_k\in\Sing(u_{t_k}) \hem\text{where $u_{t_k}\in\mathcal{M}[\varphi_{t_k}]$, and }
\omega(x_k,t_k)\to\omega(0,0).
$$
Recall that for points in $\RedSing$, the sets of minimizers are singletons  and the map $t\mapsto \varphi_t$ is continuous (see Remark~\ref{RemContAlmEverywhere}). 

There are three cases to consider:
\begin{enumerate}
\item{There is a subsequence of $t_k=0$;}
\item{There is a subsequence of $t_k>0$; and}
\item{There is a subsequence of $t_k<0$.}
\end{enumerate}
In the first case,   the sequence $x_k\in\Sing(u_0)$ accumulates at $0\in\Sing(u_0)$, contradicting Theorem \ref{ThmRegOfFBAC} since $d=\Dac+1$. It remains to study the second and the third cases. 

Below we show how the second case leads to a contradiction. The same argument works for the third case.

Under the assumption $t_k>0$, Corollary \ref{CorNonOverLappingFBAC} implies $r_k:=|x_k|>0$. Define two rescaled families:
$$
v_k(x):=\frac{1}{r_k}u_{t_k}(x_k+r_kx)
$$
and
$$
u_k(x):=\frac{1}{r_k}u_0(r_kx),
$$
then we apply Proposition \ref{PropBlowUpAC} and Lemma \ref{LemVariableCenterAC} to see that, up to a subsequence, 
$$
v_k\to v_\infty, \text{ and }u_k\to u_\infty 
\hem\text{ locally uniformly in }\R^d,
$$
where $v_\infty$ and $u_\infty$ are minimizing cones.  

Moreover, we have by Proposition~\ref{PropStabSingAC}
$$
0\in\Sing(v_\infty)\cap\Sing(u_\infty).
$$

On the other hand, with $t_k>0$, Proposition \ref{PropComparison} gives $u_{t_k}\ge u_0$ in $B_1$. Consequently, we have
$$
v_k\left(x-\frac{x_k}{r_k}\right)=\frac{1}{r_k}u_{t_k}(r_kx)\ge\frac{1}{r_k}u_0(r_kx)=u_k(x).
$$
Since $x_k/r_k\in\Sph$, up to subsequence, we have $x_k/r_k\to y_\infty\in\Sph$. Without loss of generality, we assume $y_\infty=e_1$. The previous comparison  implies
\begin{equation}
\label{EqnIfWeShiftThenWeOrder}
v_\infty(x-e_1)\ge u_\infty(x) \text{ for all }x\in\R^d.
\end{equation}
Lemma \ref{LemHomogeneousFunctions} implies
$
v_\infty\ge u_\infty \text{ in }\R^d,
$
which gives
$$
v_\infty=u_\infty \text{ in }\R^d
$$
by Lemma \ref{LemOrderedConesAC}. 
Plug this into \eqref{EqnIfWeShiftThenWeOrder}, we see that
$$
u_\infty(\cdot-e_1)\ge u_\infty(\cdot) \text{ in }\R^d.
$$

Since $0\in\Sing(u_\infty)$, Lemma \ref{LemMonotoneConesAC} implies that $u_\infty$ is invariant in the $e_1$-direction. By Lemma \ref{LemInduceInLowerDimensionAC}, this gives a minimizer in $\R^{\Dac}$ with non-empty singular part, contradicting Theorem \ref{ThmRegOfFBAC}.
\end{proof}

\subsection{Proof of Theorem \ref{ThmGRAC} for $d\ge\Dac+2$}
In this subsection, we turn to the part of Theorem \ref{ThmGRAC} that deals with the case $d\ge\Dac+2$. 

Once we show that the set $\mathcal{S}$ from \eqref{EqnSpaceTimeSing} satisfies the two hypotheses in Lemma~\ref{LemGenericReduction} with $\gamma=d-\Dac-1$ and $s=1$, then the desired conclusion follows by applying that lemma.

 Lemma \ref{LemRetractionAC} shows that the set $\mathcal{S}$ satisfies the second hypothesis of Lemma \ref{LemGenericReduction} for with $s=1$. Therefore, it suffices to establish the first hypothesis  for $\mathcal{S}.$

This is the content of the following:
\begin{prop}
\label{PropDimensionOfSpaceTimeSingAC}
For $d\ge\Dac+2$ and a family $\FamPhi$ satisfying Assumption~\ref{AssGR}, we have
$$
\DimH(\pi_x(\mathcal{S}))\le d-\Dac-1.
$$
\end{prop}
Here $\DimH$ denotes the Hausdorff dimension, the projection $\pi_x$ is defined in \eqref{EqnSpaceTimeProjections}, the set of space-time singularities $\mathcal{S}$ is defined in \eqref{EqnSpaceTimeSing}, and the dimension $\Dac$ is from \eqref{EqnCriticalAC}.

For simplicity, we introduce a parameter $m_d$ as
\begin{equation}
\label{EqnTheParameterMdAC}
m_d:=d-\Dac-1.
\end{equation} 

We begin with the following observation:
\begin{lem}
\label{LemConeMultiMonotoneAC}
Let $u$ be a minimizing cone for the Alt-Caffarelli energy $\EnAC(\cdot)$ in~$\R^d$.  

Suppose that there is a family of directions $\{\xi_j\}_{j=1,2,\dots,m_d+1}\subset\Sph$ satisfying
$$
\mathrm{dim}(\mathrm{span}\{\xi_j\})=m_d+1,
$$
and
$$
\partial_{\xi_j} u\ge 0 \text{ in }\R^d.
$$
Then, up to a rotation,  we have
$$
u(x)=x_1^+\quad\text{in}\quad \R^d. 
$$
\end{lem} 
\begin{proof}
If we have, for one $\xi_j$ and one point $x\in\R^d$, 
$$
\partial_{\xi_j} u(x)>0,
$$
then Lemma \ref{LemMonotoneConesAC} gives the desired conclusion. 

As a result, we just need to consider the case when
$$
\partial_{\xi_j} u\equiv 0\hem \text{ in }\R^d \text{ for all }j=1,2,\dots,m_d+1.
$$
By taking linear combinations of these relations, we see that $u$ is invariant along $(m_d+1)$ orthogonal directions. Up to a rotation, we might assume 
$$
u(x_1,x_2,\dots,x_d)=u(0,0,\dots,0,x_{m_d+2},x_{m_d+3},\dots,x_d) 
\hem\text{ in }\R^d.
$$
Lemma \ref{LemInduceInLowerDimensionAC} says that $(x_{m_d+2},x_{m_d+3},\dots,x_d)\mapsto u(0,0,\dots,0,x_{m_d+2},x_{m_d+3},\dots,x_d)$ is a minimizing cone in $\R^{\Dac}$. By definition of $\Dac$ from \eqref{EqnCriticalAC}, we see that $u=x_1^+$ up to a rotation. 
\end{proof} 

Now we give the proof of Proposition \ref{PropDimensionOfSpaceTimeSingAC}.
\begin{proof}[Proof of Proposition \ref{PropDimensionOfSpaceTimeSingAC}]
Suppose the conclusion of Proposition \ref{PropDimensionOfSpaceTimeSingAC} fails, then the hypothesis of Lemma \ref{LemReifenberg} fails at some point in $\pi_x(\mathcal{S})$ for $m=m_d$.  

Below we show that this leads to a contradiction.

\vem

\noindent \textit{Step 1: The setting.}   Without of loss of generality, we assume the point of failure is 
$(0,0)\in\mathcal{S}^*$ (times not belonging here are countable, and for each of them  the dimension of the singular set is at most $d-\Dac-1$ by Theorem~\ref{ThmRegOfFBAC}). This means that there is $\eps>0$ such that for each $k\in\N$, we can find $r_k<1/k$ and points $\{x_k^{(j)}\}_{j=1,2,\dots,m_d+1}$ such that 
$$
x_k^{(j)}\in B_{r_k}, \hem |\omega(x_k^{(j)},t_k^{(j)})-\omega(0,0)|<\frac{1}{k},
$$
but
$$
\max_{j=1,2,\dots,m_d+1}\mathrm{dist}(x_k^{(j)},\Pi)\ge\eps r_k
$$
for any $m_d$-dimensional subspace $\Pi$. 

Here, for each $k\in\N$, the time instance $t_k^{(j)}$ is chosen  such that $(x_k^{(j)},t_k^{(j)})\in\mathcal{S}$. By Corollary \ref{CorNonOverLappingFBAC}, this choice is unique. By Lemma \ref{LemRetractionAC}, we can assume 
$$
t_k^{(j)}\to0 \text{ for each }j=1,2,\dots, m_d+1.
$$

Take the rescaled points
$$
y_k^{(j)}=x_k^{(j)}/r_k,
$$
then 
$$
y_k^{(j)}\in B_1\backslash B_\eps, \hem\text{ and }\max_{j=1,2,\dots,m_d+1}\mathrm{dist}(y_k^{(j)},\Pi)\ge\eps
$$
for any $m_d$-dimensional subspace $\Pi$. 

Up to a subsequence of $k\to\infty$, we have 
$$
y_k^{(j)}\to y^{(j)}\in \overline{B_1}\backslash B_\eps \hem\text{ for each }j=1,2,\dots, m_d+1
$$
and
$$
\mathrm{dim}(\mathrm{span}\{y^{(j)}\})=m_d+1.
$$

Define the rescaled function 
\[
u_k(x):=\frac{1}{r_k}u_0(r_kx).
\]
Since $(0,0)\in\mathcal{S}$, Proposition \ref{PropBlowUpAC} and Proposition \ref{PropStabSingAC} imply that, up to a subsequence, 
$$
u_k\to u_\infty \text{ locally uniformly in }\R^d,
$$
where $u_\infty$ is a minimizing cone with $0\in\Sing(u_\infty)$. We make the following claim:

\noindent \textit{Claim:} The minimizing cone $u_\infty$ is monotone along $y^{(j)}$-direction for each $j=1,\dots,{m_d+1}$.

Once this is achieved, Lemma \ref{LemConeMultiMonotoneAC} implies that $u_\infty$ is of the form $x_1^+$ (up to a rotation), contradicting  $0\in\Sing(u_\infty)$. Thus, it only remains to prove the \textit{Claim}, which we do below in two different parts. 
\vem

\noindent \textit{Step 2: The cone $u_\infty$ is  invariant  in the $y^{(1)}$-direction if there is yet a further subsequence $k\to \infty$ with $t_{k}^{(1)}=0$.}

For simplicity,  we assume $
y^{(1)}=e_1.
$
In this case, we have $y_k^{(1)}\in\Sing(u_k)$ for each $k$. If we define
$$
v_k(x)=\frac{1}{r_k}u_0(x_k^{(1)}+r_kx),
$$
then Lemma \ref{LemVariableCenterAC} implies, up to a subsequence, 
$$
v_k\to v_\infty \text{ locally uniformly in }\R^d,
$$
where $v_\infty$ is a minimizing cone. 

By definition, we have
$$
v_k(x-y_k^{(1)})=\frac{1}{r_k}u_0(r_kx)=u_k(x).
$$
Local  uniform convergences of $v_k\to v_\infty$ and $u_k\to u_\infty$, together with $y_k^{(1)}\to e_1$, give 
\begin{equation*}
v_\infty(x-e_1)=u_\infty(x) \hem\text{ in }\R^d.
\end{equation*}

Lemma \ref{LemHomogeneousFunctions} implies $v_\infty=u_\infty$. Thus
$$
u_\infty(x-e_1)=u_\infty(x) \text{ in }\R^d.
$$
Another application of Lemma \ref{LemHomogeneousFunctions} implies that $u_\infty$ is invariant in the $e_1$-direction.

\vem

\noindent \textit{Step 3: The cone $u_\infty$ is monotone in the $y^{(1)}$-direction if there is yet a further subsequence $k\to \infty$ with $t_k^{(1)}>0$.}

Again we assume $y^{(1)}=e_1$ to simplify the exposition. In this case, we take
$$
v_k(x)=\frac{1}{r_k}u_{t_k^{(1)}}(x_k^{(1)}+r_kx).
$$
Lemma \ref{LemVariableCenterAC} implies that along a subsequence 
$$
v_k\to v_\infty \text{ locally uniformly in }\R^d,
$$
where $v_\infty$ is a minimizing cone. 

Proposition \ref{PropComparison} implies $u_{t_k^{(1)}}\ge u_0$, and as a result, 
$
v_k(x-y_k^{(1)})\ge u_k(x).
$
Sending $k\to\infty$, we get $v_\infty(\cdot-e_1)\ge u_\infty(\cdot)$. Lemma \ref{LemHomogeneousFunctions} implies
$v_\infty\ge u_\infty$. From here we invoke Lemma \ref{LemOrderedConesAC} to conclude $v_\infty=u_\infty$, which in turn implies $u_\infty(\cdot-e_1)\ge u_\infty(\cdot)$.

Lemma \ref{LemHomogeneousFunctions} gives that $u_\infty$ is monotone in $y^{(1)}.$

\vem

\noindent \textit{Step 4: The cone $u_\infty$ is monotone in the $y^{(j)}$-direction for $j=1,2,\dots,m_d+1$.}

The argument in \textit{Step 3} also shows that $u_\infty$ is monotone in the $y^{(1)}$-direction if there is yet a further subsequence $k\to \infty$ with $t_k^{(1)}<0.$ Combining this with \textit{Step~2} and \textit{Step~3}, we conclude that $u_\infty$ is monotone in the $y^{(1)}$-direction. 

Similar arguments can then be applied to each $y^{(j)}$ up to taking further subsequences, to conclude the proof of the claim, and thus, of the proposition.
\end{proof}

Finally, combining the previous results we obtain the proof of Theorem~\ref{ThmGRAC}:

\begin{proof}[Proof of Theorem~\ref{ThmGRAC}]
The case $d=\Dac+1$ follows from Proposition~\ref{PropEmptySingInCriticalDimAC} (together with Theorem \ref{ThmGU} and Lemma \ref{LemPointsOfContinuity}; see~\eqref{EqnDecompositionOfSingTimes}).

The case $d\ge\Dac+2$ is a consequence of Lemma~\ref{LemGenericReduction} thanks to Proposition~\ref{PropDimensionOfSpaceTimeSingAC} and Lemma~\ref{LemRetractionAC}.
\end{proof}

And:

\begin{proof}[Proof of Corollary~\ref{CorGRAC}]
Follows directly from Theorem~\ref{ThmGRAC} (cf. the proof of Corollary~\ref{CorGU}). 
\end{proof}

\section{Generic regularity  in the Alt-Phillips problem}
\label{sec:5}
In this section, we turn our attention to the one-phase Alt-Phillips energy $\EnAP(\cdot)$ in \eqref{EqnAP} and  establish Theorem \ref{ThmGRAP}. Recall that we deal with nonnegative minimizers for $\EnAP(\cdot)$. The parameter $\gamma$ and the corresponding $\beta$ from \eqref{EqnBeta} satisfy
$$
\gamma\in(0,1), \text{ and }\beta = {\textstyle \frac{2}{2-\gamma}}\in(1,2).
$$
We have seen the general strategy in Section \ref{SecGRAC}, and for that reason,  some proofs are  sketched. New ideas and techniques, however, are required for a few crucial ingredients. 

While minimizing cones for the Alt-Caffarelli functional  solve an eigenvalue problem on the sphere,  this eigenvalue problem becomes nonlinear for the Alt-Phillips functional, and as a consequence, it is less straightforward to read information from it. Thus, the counterpart of Lemma \ref{LemOrderedConesAC}, which is the starting point for all developments in Section \ref{SecGRAC}, requires a new argument involving the Weiss energy density. 

Another missing ingredient is the theory of monotone solutions as in De Silva-Jerison \cite{DJ} (see also \cite{EFeY}). This theory is based on the NTA-estimate for $\EnAC(\cdot)$ as in De Silva \cite{D2}, which is also missing for $\EnAP(\cdot)$. In Section \ref{SecGRAC}, this theory is responsible for the classification of monotone cones as in Lemmas~\ref{LemMonotoneConesAC} and~\ref{LemConeMultiMonotoneAC}. For the Alt-Phillips problem, we use a different argument: the crucial step is to  rule out points where the free boundary is tangential to the monotone direction of the cone. To achieve this, we need a detailed expansion of a minimizer near a regular  point, which is the content of Appendix \ref{AppExpansionNearRegularPoint}.

\subsection{Preparatory Lemmas}We begin by proving some properties of minimizers $\mathcal{M}[\EnAP,\cdot]$  to the Alt-Phillips energy $\EnAP(\cdot)$ from \eqref{EqnAP} (recall \eqref{EqnSetOfMin}).

The following lemma allows  blowing up along variable centers in space-time. It is the analogue of Lemma~\ref{LemVariableCenterAC} in this context. Recall the  space-time free boundary  $\mathcal{G}$ in \eqref{EqnSpaceTimeFreeBoundary},  the Weiss energy $W(\cdot)$ from \eqref{EqnWeissAP} and the Weiss energy density $\omega(\cdot,\cdot)$ from \eqref{EqnWeissLimit}.

\begin{lem}
\label{LemVariableCenterAP}
For a family $\FamPhi$ satisfying Assumption \ref{AssGR}, suppose that $u_t\in\M[\EnAP, \varphi_t]$ for each $t\in(-1,1)$.

For a sequence $(x_k,t_k)\in\mathcal{G}$ satisfying 
$$
x_k\neq0, \hem (x_k,t_k)\to (0,0)\in\mathcal{G} \text{ and }\omega(x_k,t_k)\to\omega(0,0),
$$
define $r_k=|x_k|$ and 
$$
u_k(x):=\frac{1}{r_k^\beta}u_{t_k}(x_k+r_kx).
$$

If $\mathcal{M}[\varphi_0]$ is a singleton and $0\notin I_D$ (see Remark~\ref{RemContAlmEverywhere}), then, up to a subsequence, we have
$$
u_k\to u_\infty \text{ locally uniformly in }\R^d,
$$
where $u_\infty$ is a minimizing cone with 
$$
W(u_\infty, 0,1)=\omega(0,0).
$$
\end{lem}

\begin{proof}
Compactness of the family $\{u_k\}$ follows from Proposition \ref{PropBasicAP}. 

Using Lemma~\ref{LemDiriConvAP}, Theorem \ref{ThmWeissAP},  the continuity of $t\mapsto \varphi_t$ at $t = 0$, as well as the hypothesis that $\mathcal{M}[\varphi_0]$ is a singleton, we know that the limit $u_\infty$ satisfies
$$
W(u_\infty,0,r)=\omega(0,0)\hem \text{ for all }r>0
$$
with the same argument as in the proof of Lemma \ref{LemVariableCenterAC}.

Theorem \ref{ThmWeissAP} implies that the limit $u_\infty$ is homogeneous with 
$
W(u_\infty, 0,1)=\omega(0,0).
$
\end{proof} 

Next we give the counterpart of Lemma \ref{LemOrderedConesAC}. Due to the nonlinearity of the Alt-Phillips problem, we are not able to establish the result with the same level of generality. Fortunately for our purpose, Lemma \ref{LemVariableCenterAP} states that  we only consider minimizing cones with the same Weiss energy density: 

\begin{lem}
\label{LemOrderedConesAP}
Suppose that $u$ and $v$ are minimizing cones for $\EnAP(\cdot)$ in $\R^d$ with 
$$
u\ge v \text{ in }\R^d
\hem\text{ and }\hem
W(u,0,1)=W(v,0,1),
$$
then
$$
u=v \text{ in }\R^d.
$$
\end{lem} 
\begin{proof}
With Lemma \ref{LemWeissRepAP}, we see that
$
\int_{B_1}u^\gamma=\int_{B_1}v^\gamma.
$
The desired conclusion follows from the ordering $u\ge v$. 
\end{proof} 

The free boundary retracts linearly with respect to the parameter $t\in(-1,1)$ (cf. Lemma~\ref{LemRetractionAC}):

\begin{lem}
\label{LemRetractionAP}
For a family $\FamPhi$ satisfying Assumption \ref{AssGR}, suppose that for each $t\in (-1, 1)$, $u_t\in\M[\EnAP, \varphi_t]$, and let $\tau \in (-1, 1)$ fixed. Suppose, also, that $x_0\in B_{1/2}\cap\Gamma(u_{t_0})$ for some $t_0\in[\tau,1)$. 

Then, there is a constant $\kappa>0$ depending only on $\tau$, $d$, $\gamma$, as well as $L$ and $\theta$ from Assumption \ref{AssGR}, such that 
$$
B_{\kappa\theta(t-t_0)}(x_0)\subset\{u_t>0\}
$$
for $t\in [t_0, 1)$.
\end{lem} 
  \begin{proof}
As in the proof for Lemma \ref{LemRetractionAC}, it suffices to show 
$$
u_{t_0}(x-se_1)\le u_t(x) \text{ in }B_{3/4}
$$
for all $0<s<\eta(t-t_0)$ with $\eta=\eta(\tau, d, \gamma, L,\theta)$.

Similar to \textit{Step 1} in the proof of Lemma \ref{LemRetractionAC}, we see that there is a small $\delta=\delta(\tau,d,\gamma,L,\theta)>0$ such that 
$$
u_{t_0}>0 \text{ in }B_1\backslash\overline{B_{1-8\delta}}.
$$
Proposition \ref{PropComparison} implies $u_t\ge u_{t_0}$, and thanks to  Proposition~\ref{PropBasicAP}, we have
$$
\Delta(u_{t}-u_{t_0})=\frac{\gamma}{2}(u_t^{\gamma-1}-u_{t_0}^{\gamma-1})\le 0 \hem\text{ in }B_{1}\backslash\overline{B_{1-8\delta}},
$$
where we used our assumption $\gamma\in(0,1)$.

As a result, the function $h$ in \textit{Step 1} from the proof of Lemma \ref{LemRetractionAC} works as a lower barrier for the difference $(u_t-u_{t_0})$, and gives a separation of between $u_t$ and $u_{t_0}$ in $B_1\setminus B_{1-4\delta}$ that is proportional to $(t-t_0)$.

By the regularity of $u_t$ and $u_{t_0}$ in $\overline{B_{1-\delta}}$ from Proposition \ref{PropBasicAP}, this separation between $u_t$ and $u_{t_0}$ is translated into an ordering between $u_t$ and translations of $u_{t_0}$, with the same argument as in the proof of Lemma \ref{LemRetractionAC}.
\end{proof}

The non-overlapping of free boundaries follows (cf. Corollary~\ref{CorNonOverLappingFBAC}):
\begin{cor}
\label{CorNonOverLappingFBAP}
Suppose the family $\FamPhi$ satisfies Assumption \ref{AssGR} and $u_t\in\M[\EnAP,\varphi_t]$ for each $t\in(-1,1)$. Then
$$
\Gamma(u_t)\cap\Gamma(u_s)=\emptyset \text{ if }t\neq s.
$$
\end{cor}

For our purpose, it is crucial to rule out monotone minimizing cones with a free boundary that becomes tangential to the monotone direction:
\begin{lem}
\label{LemTangentialFBAP}
Suppose that $u$ is a minimizing cone for the Alt-Phillips energy $\EnAP(\cdot)$ in $\R^d$ with 
$$\GU\cap\partial B_1\subset\RedGU$$
and
$$
\partial_1u\ge0 \text{ in }\R^d.
$$
If at some $p\in\GU\cap\partial B_1$  we have
$$
\nu_p\cdot e_1=0,
$$
then
$$
\partial_1u\equiv 0\hem \text{ in }\R^d.
$$
\end{lem} 
Recall  the free boundary $\GU$ from \eqref{EqnFreeBoundary}, its regular part $\RedGU$ from Definition~\ref{DefRegSingAP}, as well as   the normal vector $\nu_p$  in \eqref{EqnUnitNormal}.

\begin{proof}
Proposition \ref{PropBasicAP} implies  that $\partial_1u$ satisfies
$$
\Delta\partial_1u=\frac{\gamma}{2}(\gamma-1)u^{\gamma-2}\partial_1u 
\hem\text{ in }\PosS.
$$
Since $\gamma\in(0,1)$, our assumption $\partial_1u\ge0$ gives
$$
\Delta\partial_1u\le0 \text{ in }\PosS.
$$
In particular, if $\partial_1u\not\equiv 0$, the strong maximum principle implies 
$\partial_1u>0$ in $\PosS$.

  By Hopf's lemma, since $\partial_1 u$ is continuous and vanishes on the free boundary, and $\Gamma(u)$ is $C^{1,\alpha}$ around $p$, we deduce that for some $c > 0$,
\begin{equation}
\label{EqnALowerBoundForPartial}
\partial_1u(p+r\nu_p)\ge cr\hem \text{ for all small } r>0.
\end{equation}

We now consider the function 
$$
w:=u^{1/\beta}.
$$
By Corollary \ref{CorAppExpansionOfGradient}, we have that for each $\alpha\in(0,1)$, there is a constant $C$ such that 
$$
|\nabla w(p+r\nu_p)-\nu_p/\beta|\le Cr^\alpha \text{ for all small }r>0.
$$
Since  $\nu_p\cdot e_1=0$, this implies
$$
\partial_1w(p+r\nu_p)\le Cr^\alpha \text{ for all small }r>0.
$$
Thus we can bound $\partial_1 u$ from above as
$$
\partial_1 u(p+r\nu_p)=\beta w^{\beta-1}\partial_1 w(p+r\nu_p)\le Cr^{\beta-1}r^\alpha \text{ for all small }r>0.
$$
If we choose $\alpha\in (0, 1)$ such that 
$$
\alpha+\beta>2,
$$
this contradicts \eqref{EqnALowerBoundForPartial}.
\end{proof}

\subsection{Proof of Theorem \ref{ThmGRAP} for $d=\Dap+1$}
We can now proceed with the proof of Theorem \ref{ThmGRAP}. In this subsection, we deal with the case $d=\Dap+1$ (recall  \eqref{EqnCriticalAP}). 

Thanks to the decomposition in \eqref{EqnDecompositionOfSingTimes}, for   $d=\Dap+1$,  it suffices to establish the following:
\begin{prop}
\label{PropEmptySingInCriticalDimAP}
For $d=\Dap+1$ and a family $\FamPhi$ satisfying Assumption~\ref{AssGR}, we have
$$
\mathcal{S}^*=\emptyset,
$$
where $\mathcal{S}^*$ is the reduced set of space-time singularities from \eqref{EqnFurtherReducedSing}.
\end{prop} 

To this end, we need to rule out non-trivial monotone cones:
\begin{lem}
\label{LemMonotoneConesCriticalDimAP}
Suppose that $u$ is a minimizing cone for $\EnAP(\cdot)$ in $\R^{\Dap+1}$ with
$$
\partial_1u\ge0 \hem\text{ in }\R^{\Dap+1}.
$$
Then, up to a rotation, 
$$
u(x)=(x_1^+/\beta)^\beta.
$$
\end{lem} 
\begin{proof}
Denote the cone of monotone directions of $u$ by $\mathcal{C}$, that is, 
$$
\mathcal{C}:=\{e\in\partial B_1:\hem \partial_eu\ge0 \hem\text{ in }\R^{\Dap+1}\}.
$$
Then $\mathcal{C}$ is a non-empty proper subset of $\partial B_1$. Consequently, we can find 
\begin{equation}
\label{EqnEStarIsInTheBoundary}
e^*\in\partial\mathcal{C}.
\end{equation} 
Proposition \ref{PropBasicAP} implies that $\mathcal{C}$ is closed. As a result, 
$\partial_{e^*}u\ge 0 \text{ in }\R^{\Dap+1}.$ 

By the strong maximum principle as in the beginning of the proof of Lemma~\ref{LemTangentialFBAP}, $\partial_{e^*}u$  cannot vanish  in  $\PosS$ unless it is identically zero, in which case Lemma~\ref{LemAppLowerDimension} and the definition of $\Dap$ imply that $u$ is of the desired form.

As a result, it suffices to consider the case when 
$$
\partial_{e^*}u\ge 0 \text{ in }\R^{\Dap+1}, \text{ and }\partial_{e^*}u>0 \text{ in }\PosS.
$$
In particular, we have
$$
\nu_p\cdot e^*\ge 0 \hem\text{ for all }p\in\GU\cap\partial B_1.
$$
(Note that Theorem \ref{ThmSingAP} implies that $\GU\backslash\{0\}\subset\RedGU$.) We claim that:

\textit{Claim: }There is $p\in\GU\cap\partial B_1$ such that 
$$
\nu_p\cdot e^*=0.
$$

Once the \textit{Claim} is established, we apply Lemma \ref{LemTangentialFBAP} to get $\partial_{e^*}u\equiv 0$ in $\R^{\Dap+1}$. Again, by Lemma \ref{LemAppLowerDimension} and the definition of $\Dap$, we are done.

Let us show the \textit{Claim}:

Suppose by contradiction that $\nu_p\cdot e^*>0$ for all $p\in\GU\cap\partial B_1$. The $C^1$-regularity of $\GU\cap\partial B_1$ implies that we can find $\delta>0$ such that 
$$
\nu_p\cdot e^*\ge\delta \hem\text{ for all } p\in\GU\cap\partial B_1.
$$
For the function $w$ defined as
$$
w:=u^{1/\beta},
$$
Corollary \ref{CorAppExpansionOfGradient} implies that, for a small $\rho>0$, 
$$
\partial_{e^*}w\ge \delta/2\hem \text{ on }\{x\in\PosS\cap\partial B_1:\hem \mathrm{dist}(x,\GU)<\rho\}.
$$
With the Lipschitz regularity of $w$, this gives a small $\eps>0$ such that 
$$
\partial_{e^*}w\ge \eps|\nabla w| \hem\text{ on }\{x\in\PosS\cap\partial B_1:\hem \mathrm{dist}(x,\GU)<\rho\}.
$$
In terms of the original function $u$, we have
\begin{equation*}
\partial_{e^*}u\ge \eps|\nabla u| \hem\text{ on }\{x\in\PosS\cap\partial B_1:\hem \mathrm{dist}(x,\GU)<\rho\}.
\end{equation*} 

In the compact complementary region $ \{x\in\PosS\cap\partial B_1:\hem \mathrm{dist}(x,\GU)\ge\rho\}$, since $\partial_{e^*}u>0$ and $u$ is $C^1$ (Proposition~\ref{PropBasicAP}), we have
$\partial_{e^*}u\ge \eps|\nabla u|$ by choosing a possibly smaller $\eps$. 

Therefore, we have
$$
\partial_{e^*}u\ge \eps|\nabla u|\hem \text{ in }\R^{\Dap+1},
$$
since $u$ is homogeneous and both sides vanish in $\ConS$.
 
From here, we see that for $e\in\partial B_1$, we have
$$
\partial_eu=\partial_{e^*}u+\nabla u\cdot(e-e^*)\ge0
$$
if $|e-e^*|<\eps$. 
This contradicts the fact that $e^*\in\partial\mathcal{C}$.
\end{proof} 

We now give the proof of Proposition \ref{PropEmptySingInCriticalDimAP}:
\begin{proof}[Proof of Proposition \ref{PropEmptySingInCriticalDimAP}]
The strategy is similar to the proof of Proposition \ref{PropEmptySingInCriticalDimAC}. We sketch the argument.

Suppose this proposition is not true,  and assume $(0,0)\in\RedSing.$ That is, 
$$
0\in\Sing(u_0) \text{ for }u_0\in\mathcal{M}[\varphi_0],
$$
and there is a sequence $(x_k,t_k)\to(0,0)$ such that  
$$
x_k\in\Sing(u_{t_k}) \hem\text{where $u_{t_k}\in\mathcal{M}[\varphi_{t_k}]$, and }
\omega(x_k,t_k)\to\omega(0,0).
$$

There are three cases to consider:
\begin{enumerate}
\item{There is a subsequence of $t_k=0$;}
\item{There is a subsequence of $t_k>0$; and}
\item{There is a subsequence of $t_k<0$.}
\end{enumerate}
In the first case, the sequence $x_k\in\Sing(u_0)$ accumulates at $0\in\Sing(u_0)$, contradicting Theorem \ref{ThmSingAP} since $d=\Dap+1$. It remains to study the other two cases. 

Below we show how the second case leads to a contradiction. The same argument works for the third case.

Under the assumption $t_k>0$, Corollary \ref{CorNonOverLappingFBAP} implies $r_k:=|x_k|>0$. We consider the following two rescaled families:
$$
v_k(x):=\frac{1}{r_k^\beta}u_{t_k}(x_k+r_kx)
\hem\text{ and }\hem
u_k(x):=\frac{1}{r_k^\beta}u_0(r_kx).
$$
By Proposition \ref{PropBlowUpAP}, Proposition \ref{PropStabSingAP} and Lemma \ref{LemVariableCenterAP}, we have, up to a subsequence, 
$$
v_k\to v_\infty, \text{ and }u_k\to u_\infty 
\hem\text{ locally uniformly in }\R^d,
$$
where $v_\infty$ and $u_\infty$ are minimizing cones with
$$
0\in\Sing(v_\infty)\cap\Sing(u_\infty)
$$
and
$$
W(v_\infty,0,1)=W(u_\infty, 0,1).
$$

With $t_k>0$, Proposition \ref{PropComparison} gives $u_{t_k}\ge u_0$ in $B_1$. Consequently, we have
\begin{equation*}
v_\infty(x-e_1)\ge u_\infty(x) \text{ for all }x\in\R^d,
\end{equation*}
where we assume, without loss of generality, that $x_k/r_k\to e_1$.

By Lemma \ref{LemHomogeneousFunctions}, we have
$
v_\infty\ge u_\infty.
$
Lemma \ref{LemOrderedConesAP} implies $v_\infty=u_\infty$. Therefore,
$
u_\infty(\cdot-e_1)\ge u_\infty(\cdot).
$
Lemma \ref{LemHomogeneousFunctions} implies that 
$$
\partial_{-e_1}u_\infty\ge 0 \text{ in }\R^d.
$$
By Lemma \ref{LemMonotoneConesCriticalDimAP}, we obtain that $u_\infty=[(-x_1)^+/\beta]^\beta$, contradicting $0\in\Sing(u_\infty)$.
\end{proof}

\subsection{Proof of Theorem \ref{ThmGRAP} for $d\ge\Dap+2$}
In this subsection, we prove the   Theorem \ref{ThmGRAP} in dimensions $d\ge\Dap+2$. Thanks to Lemmas~\ref{LemRetractionAP} and~\ref{LemGenericReduction},  it suffices to establish the following proposition (cf. Proposition~\ref{PropDimensionOfSpaceTimeSingAC}). We recall that projection $\pi_x(\cdot)$ is defined in \eqref{EqnSpaceTimeProjections}, the set of space-time singularities $\mathcal{S}$ is defined in \eqref{EqnSpaceTimeSing}, and the dimension $\Dap$ is from \eqref{EqnCriticalAP}.

\begin{prop}
\label{PropDimensionOfSpaceTimeSingAP}
For $d\ge\Dap+2$ and a family $\FamPhi$ satisfying Assumption~\ref{AssGR}, we have
$$
\DimH(\pi_x(\mathcal{S}))\le d-\Dap-1.
$$
\end{prop} 

For simplicity, we introduce a parameter $m_d$ as
\begin{equation}
\label{EqnTheParameterMdAP}
m_d:=d-\Dap-1.
\end{equation} 
We begin by ruling out cones that are monotone in multiple directions (cf. Lemma~\ref{LemConeMultiMonotoneAC}):

\begin{lem}
\label{LemConeMultiMonotoneAP}
Let $u$ be a minimizing cone for the Alt-Phillips energy $\EnAP(\cdot)$ in $\R^d$.  

Suppose that there is a family of directions $\{\xi_j\}_{j=1,2,\dots,m_d+1}\subset\Sph$ satisfying
$$
\mathrm{dim}(\mathrm{span}\{\xi_j\})=m_d+1,
$$
and
$$
\partial_{\xi_j} u\ge 0 \text{ in }\R^d.
$$
Then, up to a rotation,  we have
$$
u=(x_1^+/\beta)^\beta.
$$
\end{lem} 
\begin{proof}
The proof is based on an induction on the dimension $d$. The base case $d=\Dap+1$ is the content of Lemma \ref{LemMonotoneConesCriticalDimAP}. Assume that we have established the result in dimensions $d=\Dap+1,\dots, n-1$, we proceed to prove the case when $d=n$.

Take $p\in\GU\cap\partial B_1$ and define
$$
u_r(x):=\frac{1}{r^\beta}u(p+rx),
$$
then with Proposition \ref{PropBlowUpAP} and Lemma \ref{LemConeSplittingAP}, we see that, up to a subsequence of $r_k\to0$, we have
$$
u_{r_k}\to u_0 \text{ locally uniformly in }\R^n,
$$
where $u_0$ is a minimizing cone independent of the $p$-direction. Lemma \ref{LemAppLowerDimension} implies that $u_0$ induces a minimizing cone in $\R^{n-1},$ denoted by $\overline{u_0}$.

Meanwhile, we have $\partial_{\xi_j} u_{r}\ge0$ for each $j=1,2,\dots,m_n+1$, implying the same property for the limit $u_0$. After reducing the dimension by $1$, the minimizer $\overline{u}_0$ is monotone along at least $m_n$ independent directions. By the result in dimension $(n-1)$, we see that $\overline{u}_0$ is a rotation of $(x_1^+/\beta)^\beta$. The limit $u_0$ is of the same form.

Therefore, we have $p\in\RedGU$ by Definition \ref{DefRegSingAP} for all $p\in\GU\cap\partial B_1$.  Theorem~\ref{ThmRegSmoothAP} implies that $\GU\cap\partial B_1$ is $C^1$. From here, the same argument as in the proof of Lemma \ref{LemMonotoneConesCriticalDimAP} gives the desired conclusion. 
\end{proof}

Now we give the proof of Proposition \ref{PropDimensionOfSpaceTimeSingAP}.

\begin{proof}[Proof of Proposition \ref{PropDimensionOfSpaceTimeSingAP}]
The proof is  the same as the proof of Proposition~\ref{PropDimensionOfSpaceTimeSingAC}, where instead of using Corollary~\ref{CorNonOverLappingFBAC}, Lemmas~\ref{LemRetractionAC}, \ref{LemConeMultiMonotoneAC}, \ref{LemVariableCenterAC}, and~\ref{LemOrderedConesAC},   Propositions~\ref{PropBlowUpAC} and~\ref{PropStabSingAC}, and Theorem~\ref{ThmRegOfFBAC}, we use, respectively, Corollary~\ref{CorNonOverLappingFBAP}, Lemmas~\ref{LemRetractionAP}, \ref{LemConeMultiMonotoneAP}, \ref{LemVariableCenterAP}, and~\ref{LemOrderedConesAP}, Propositions~\ref{PropBlowUpAP} and~\ref{PropStabSingAP}, and Theorem~\ref{ThmSingAP}.
\end{proof}

Finally, combining the previous results we obtain the proof of Theorem~\ref{ThmGRAP}:

\begin{proof}[Proof of Theorem~\ref{ThmGRAP}]
The case $d=\Dap+1$ follows from Proposition~\ref{PropEmptySingInCriticalDimAP} (together with Theorem \ref{ThmGU} and Lemma \ref{LemPointsOfContinuity}; see~\eqref{EqnDecompositionOfSingTimes}).

The case $d\ge\Dap+2$ is a consequence of Lemma~\ref{LemGenericReduction} thanks to Proposition~\ref{PropDimensionOfSpaceTimeSingAP} and Lemma~\ref{LemRetractionAP}.
\end{proof}

 And:

\begin{proof}[Proof of Corollary~\ref{CorGRAP}]
Follows directly from Theorem~\ref{ThmGRAP} (cf. the proof of Corollary~\ref{CorGU}). 
\end{proof}

\appendix

\section{Singular sets in the Alt-Phillips problem}
\label{AppSingAP}

In this appendix, we establish Theorem \ref{ThmSingAP} about dimensions of singular sets in the one-phase Alt-Phillips problem. Since the strategy is standard, certain parts of the proof are only sketched. For the one-phase Alt-Caffarelli problem, similar results were proved  in \cite{W}. The proofs in this appendix are adaptations of those in \cite[Chapter 10]{V}.

Recall the one-phase Alt-Phillips energy $\EnAP(\cdot)$ from \eqref{EqnAP}, the set of minimizers $\mathcal{M}[\cdot,\cdot]$ from \eqref{EqnSetOfMin}, as well as the dimension $\Dap$ from \eqref{EqnCriticalAP}. For a minimizer $u$, the free boundary $\GU$ is defined in \eqref{EqnFreeBoundary}. It  decomposes as $\GU=\RedGU\cup\SingU$ according to Definition \ref{DefRegSingAP}.

Suppose that $u\in\mathcal{M}[u]$ in $B_1\subset\R^d$ for $d\le\Dap.$ 
For a free boundary point $x_0\in\GU$, Proposition \ref{PropBlowUpAP} implies that the rescaled functions $u_{x_0,r}$ converge, along a subsequence of $r_k\to 0$, to a minimizing cone $u_{x_0}$. By the definition of $\Dap$, the limit $u_{x_0}$ is, up to a rotation, $(x_1^+/\beta)^\beta$ (see \eqref{EqnBeta} for the constant $\beta$). According to Definition \ref{DefRegSingAP}, this implies that $x_0\in\RedGU$. 

Consequently, we have 
\begin{equation}
\label{EqnAppEmptySing}
\SingU=\emptyset\hem \text{ if }d\le\Dap.
\end{equation}
To establish Theorem \ref{ThmSingAP}, it suffices to prove the following two propositions:
\begin{prop}
\label{PropAppCriticalDim}
Let $u\in\mathcal{M}[u]$  with $d=\Dap+1$. Then 
$
\SingU
$
is locally discrete.
\end{prop} 
\begin{prop}
\label{PropAppGeneralDim}
Let $u\in\mathcal{M}[u]$ with $d\ge\Dap+2$. Then 
$
\DimH(\SingU)\le d-\Dap-1.
$
\end{prop}

We begin with some preparatory results.
\begin{lem}
\label{LemAppUSCHausM}
For $s>0$, let $\mathcal{H}^s(\cdot)$ denotes the $s$-dimensional Hausdorff measure. 

Suppose that the sequence $u_k\in\mathcal{M}[u_k]$  converge to $u\in\mathcal{M}[u]$ locally uniformly in $B_1$. 
Then
$$
\mathcal{H}^s(\SingU\cap\overline{B_{1/2}})\ge\limsup\mathcal{H}^s(\mathrm{Sing}(u_k)\cap\overline{B_{1/2}}).
$$
\end{lem} 

\begin{proof} 
By the definition of Hausdorff measures, it suffices to prove the following:

\textit{Claim:} Suppose that we have a collection of open balls $\{B^j\}_{j=1,2,\dots,N}$ such that 
$$
\bigcup_j B^j\supset\SingU\cap\overline{B_{1/2}},
$$
then 
$$
\bigcup_j B^j\supset\mathrm{Sing}(u_k)\cap\overline{B_{1/2}} \text{ for all large }k.
$$

Suppose not, we find $x_k\in\mathrm{Sing}(u_k)\cap\overline{B_{1/2}}$ but $x_k\not\in\bigcup_j B^j.$
Up to a subsequence, we have $x_k\to x_\infty\in\overline{B_{1/2}}$. Openness of $B^j$'s implies that $x_\infty\not\in\bigcup_j B^j$. 

Meanwhile, Proposition \ref{PropStabSingAP} gives $x_\infty\in\SingU$, contradicting our assumption that $\bigcup_j B^j\supset\SingU\cap\overline{B_{1/2}}.$
\end{proof} 

The following lemma says that a minimizer, if independent of one variable, induces a minimizer in a lower dimensional space.
\begin{lem}
\label{LemAppLowerDimension}
Suppose that $u$ minimizes the energy $\EnAP(\cdot)$ in $\R^d$ and satisfies
$$
u(x',x_d)=u(x',0) \text{ for all }x_d\in\R.
$$

Define $\bar{u}:\R^{d-1}\to\R$ by $\bar{u}(x')=u(x',0)$, then $\bar{u}$ minimizes $\EnAP(\cdot)$ in $\R^{d-1}$.
\end{lem} 
\begin{proof} 
Suppose not, we find $R>0$ and $\bar{v}$ such that 
$$\bar{v}=\bar{u}\text{ outside }B_R'\subset\R^{d-1},$$
and 
$$
\int_{B_R'}\left( |\nabla\bar{u}|^2+\bar{u}^\gamma\right) \ge\int_{B_R'}(|\nabla\bar{v}|^2+\bar{v}^\gamma)+\delta
$$
for some $\delta>0$.

Given $T>0$ large, we denote by $\eta_T\ge 0$ a one-dimensional cut-off function such that $\eta_T=1$ on $[-T,T]$, $\eta_T=0$ outside $(-T-1,T+1)$, and $|\eta_T'|\le 2$ on $\R$. The following function
$$
v(x',x_d):=\bar{v}(x')\eta_T(x_d)+\bar{u}(x')(1-\eta_T(x_d))
$$
agrees with $u$ outside $\Omega_T:=B_R'\times(-T-1,T+1)$.

A direct computation gives that 
$$
\int_{\Omega_T}\left(|\nabla u|^2+u^\gamma\right) =(2T+2)\int_{B_R'}\left(|\nabla\bar{u}|^2+\bar{u}^\gamma\right),
$$
and
$$
\int_{\Omega_T}\left(|\nabla v|^2+v^\gamma\right)= 2T \int_{B_R'}(|\nabla\bar{v}|^2+\bar{v}^\gamma)+I,$$
where 
$$
I\le C \int_{B_R'}(|\nabla\bar{u}|^2+|\nabla\bar{v}|^2+\bar{u}^\gamma+\bar{v}^\gamma+\bar{u}^2+\bar{v}^2)
$$
for a dimensional constant $C$.

As a result, 
\begin{align*}
\int_{\Omega_T}(|\nabla u|^2+u^\gamma)-\int_{\Omega_T}(|\nabla v|^2+v^\gamma)&\ge 2T\left[\int_{B_R'}(|\nabla\bar{u}|^2+\bar{u}^\gamma)-\int_{B_R'}(|\nabla\bar{v}|^2+\bar{v}^\gamma)\right]-I\\
&\ge 2T\delta-I.
\end{align*}
Choosing $T$ large enough, we have $\int_{\Omega_T}(|\nabla u|^2+u^\gamma)>\int_{\Omega_T}(|\nabla v|^2+v^\gamma)$, contradicting the minimizing property of $u$ in $\Omega_T$.
\end{proof} 

\vem

With these preparations, we give the proof of Proposition \ref{PropAppCriticalDim}:
\begin{proof}[Proof of Proposition \ref{PropAppCriticalDim}]  
Suppose the statement is false. Without loss of generality, we   assume that $0\in\SingU$ and that there is a sequence $x_k\neq 0$ such that $x_k\in\SingU$ and $x_k\to 0$. 

Define $r_k:=|x_k|$, $y_k:=x_k/r_k$, and 
$$
u_k(x):=\frac{1}{r_k^\beta}u(r_kx).
$$
Then $y_k\in\mathrm{Sing}(u_k)$ for each $k$. 

By Proposition \ref{PropBlowUpAP}, up to a subsequence, the sequence $u_k$ converge locally uniformly to a minimizing cone $u_\infty$. Since $|y_k|=1$ for all $k$, up to a subsequence, $y_k\to y_\infty\in\Sph$. Without loss of generality, we assume $y_\infty=e_1$. Proposition \ref{PropStabSingAP} implies that
$
e_1\in\mathrm{Sing}(u_\infty).
$

Define the rescaled function
$$
v_r(x):=\frac{1}{r^\beta}u(e_1+rx).
$$
Proposition \ref{PropBlowUpAP} implies that, up to a subsequence, $v_r\to v_0$, a minimizing cone, and Lemma \ref{LemConeSplittingAP} gives that $v_0$ is constant along the $e_1$-direction. 

Let $\bar{v}_0(x_2,x_3,\dots,x_d)=v_0(0,x_2,x_3,\dots,x_d)$. By Lemma \ref{LemAppLowerDimension}, $\bar{v}_0$ is a minimizer in $\R^{\Dap}$.
Since $0\in\mathrm{Sing}(v_r)$ for each $r>0$, we have $0\in\mathrm{Sing}(\bar{v}_0)$, contradicting~\eqref{EqnAppEmptySing}.
\end{proof} 

Below we give
\begin{proof}[Proof of Proposition \ref{PropAppGeneralDim}]
We prove the result by an induction on the dimension $d$, with the base case provided by Proposition \ref{PropAppCriticalDim}. Assuming, for $d=1,2,\dots,n-1$, the statement in the proposition is true, we will prove the estimate for $d=n$.

Suppose, on the contrary, that 
\begin{equation}
\label{EqnAppDefOfm}
\DimH(\SingU)>n-\Dap-1=:m
\end{equation} for some $u\in\mathcal{M}[u]$ in $B_1\subset\R^n$, then we find $s>0$ such that 
$$
\mathcal{H}^{m+s}(\SingU)>0.
$$
Without loss of generality, we can assume by Lemma \ref{LemExistenceOfDensityPoints} that
$$
\frac{\mathcal{H}^{m+s}(\SingU\cap B_{r_k})}{r_k^{m+s}}\ge\delta>0
$$
along a sequence $r_k\to0$.

For the rescaled functions
$
u_k(x):=\frac{1}{r_k^\beta}u(r_kx),
$
we have
$$
\mathcal{H}^{m+s}(\mathrm{Sing}(u_k)\cap \overline{B_1})\ge\delta.
$$
Along a subsequence, Proposition \ref{PropBlowUpAP} implies $u_k\to u_\infty$, a minimizing cone. Lemma \ref{LemAppUSCHausM} leads to 
$$
\mathcal{H}^{m+s}(\mathrm{Sing}(u_\infty)\cap \overline{B_1})\ge\delta.
$$

\vem

Since we are considering the case when $m> 0$, this implies $\mathcal{H}^{m+s}(\mathrm{Sing}(u_\infty)\cap \overline{B_1}\backslash\{0\})\ge\delta$. Lemma \ref{LemExistenceOfDensityPoints} implies, for some $p\neq 0$, 
\begin{equation}
\label{EqnAppALowerBound}
\frac{\mathcal{H}^{m+s}(\mathrm{Sing}(u_\infty)\cap B_{r_k}(p))}{r_k^{m+s}}\ge\delta>0
\end{equation}
along a sequence $r_k\to0$. With homogeneity of $u_\infty$, we assume $p=e_1$.

If we take the rescaled functions 
$$
v_k(x):=\frac{1}{r_k^\beta}u_\infty(e_1+r_kx),
$$
then the lower bound \eqref{EqnAppALowerBound}, together with Proposition \ref{PropBlowUpAP}, Lemmas~\ref{LemConeSplittingAP}, \ref{LemAppUSCHausM}, and~\ref{LemAppLowerDimension}, imply that
$$
\mathcal{H}^{m+s-1}(\mathrm{Sing}(\bar{v}))>0
$$
for a minimizer $\bar{v}$ in $\R^{n-1}$. This contradicts our induction hypothesis. 
\end{proof}

\section{Behavior of minimizers near a regular point in the Alt-Phillips problem}
\label{AppExpansionNearRegularPoint}

Suppose that $u$ is a minimizer of the one-phase Alt-Phillips functional  \eqref{EqnAP} with  the parameter $\gamma$ and the corresponding $\beta$ from \eqref{EqnBeta} satisfying
$$
\gamma\in(0,1) \text{ and }\beta = {\textstyle\frac{2}{2-\gamma}}\in(1,2).
$$
To apply similar techniques from the Alt-Caffarelli problem, Alt-Phillips \cite{AP} studied the  function 
\begin{equation}
\label{EqnTransformU}
w:=u^{1/\beta}.
\end{equation}
In this section, we are interested in the behavior of $u$ near a regular point on the free boundary. Around such a point,  Theorem \ref{ThmRegSmoothAP} implies that  the free boundary is a $C^{1,\alpha}$-hypersurface. In this setting,  it is not difficult to see that $w$ satisfies
\begin{equation}
\label{EqnSystemForTransformedU}
\begin{cases}
\Delta w=\frac{\frac{\gamma}{2\beta}-(\beta-1)|\nabla w|^2}{w} &\text{ in }\{w>0\},\\
|\nabla w|= 1/\beta&\text{ on }\partial\{w>0\}
\end{cases}
\end{equation}
in the viscosity sense as in De Silva-Savin \cite{DS}.

In De Silva-Savin \cite{DS}, the authors developed the framework to study similar problems with a large class of nonlinearities. Among their results is the following expansion  near a `flat point'. In our context, it reads:
\begin{prop}[By iterations of Proposition 6.1 of De Silva-Savin \cite{DS}]
\label{PropThePropByDS}
Suppose that $w$ is a solution to \eqref{EqnSystemForTransformedU} with $0\in\partial\{w>0\}$. If, for some small $\eps>0$, we have
$$
(x_d/\beta-\eps)^+\le w\le (x_d/\beta+\eps)^+ \text{ in }B_1,
$$
then, for some $e\in\Sph$ with $|e-e_d|\le C\eps$, 
we have
$$
|w-(x\cdot e/\beta)^+|\le C\eps r^{1+\alpha} \text{ in }B_r
$$
for all small $r\in(1/2)$ and some $\alpha\in(0,1)$.
\end{prop} 
For the class of nonlinearities studied by De Silva-Savin,  the exponent $\alpha$  depends on properties of the nonlinearity; in our context, we need to quantify it. 

Specifically for the Alt-Phillips problem with $\gamma\in(0,1)$, we show that the expansion in Proposition \ref{PropThePropByDS} holds for all $\alpha\in (0,1)$:
\begin{prop}
\label{PropSharpAlpha}
Suppose that $w$ solves \eqref{EqnSystemForTransformedU} with $0\in\partial\{w>0\}$. 

Given $\alpha\in(0,1)$, there is a small $\bar{\eps}=\bar\eps(\alpha,\gamma,d)>0$ such that if 
$$
(x_d/\beta-\eps)^+\le w\le (x_d/\beta+\eps)^+ \text{ in }B_1 \text{ for some }\eps<\bar\eps,
$$
then, for some $e\in\Sph$ with $|e-e_d|\le C\eps$, 
we have
$$
|w-(x\cdot e/\beta)^+|\le C\eps r^{1+\alpha} \text{ in }B_r
$$
for all $r\in(0,1/2)$ and a constant $C$ depending on $\alpha$, $d$ and $\gamma.$
\end{prop} 

This proposition is a consequence of the following:
\begin{lem}
\label{LemIOF}
Under the same assumptions as in Proposition \ref{PropSharpAlpha}, there are constants $A=A(\gamma,d)>1$ and $r_0=r_0(\alpha,\gamma,d)\in(0,1/2)$ such that 
$$
Ar_0<\frac12, \hem \log_{r_0}A>\alpha-1, 
$$
and
$$
(x\cdot e/\beta-Ar_0^2\eps)^+\le w\le (x\cdot e/\beta+Ar_0^2\eps)^+ \text{ in }B_{r_0}
$$
for some $e\in\Sph$ with $|e-e_d|\le C\eps.$
\end{lem} 
Indeed, once we establish Lemma \ref{LemIOF}, an iteration gives
$$
|w-(x\cdot e_k/\beta)^+|\le C\eps(Ar_0)^kr_0^k \hem\text{ in }B_{r_0^k}
$$
with $|e_k-e_{k-1}|\le C\eps(Ar_0)^k.$ In particular, there is $e\in\Sph$ such that 
$$
|e-e_k|\le C\eps(Ar_0)^k.
$$

For any $r\in(0,1/2)$, find $k\in\N$ with $r_0^{k+1}\le r<r_0^{k}$, then we have
$$
|w-(x\cdot e/\beta)^+|\le |w-(x\cdot e_k/\beta)^+|+|x\cdot(e_k-e)|\le C\eps(Ar_0)^kr_0^k\hem \text{ in }B_r.
$$
That is, in $B_r$, we have
$$
|w-(x\cdot e/\beta)^+|\le C\eps r(Ar_0)^{\log_{r_0}r}\le C\eps r\cdot r^{1+\log_{r_0}A}.
$$
With $\log_{r_0}A>\alpha-1$, we see that 
$$
|w-(x\cdot e/\beta)^+|\le C\eps r^{1+\alpha} \text{ in }B_r,
$$
which is the desired conclusion of Proposition \ref{PropSharpAlpha}.

\vem

Below we give the proof of Lemma \ref{LemIOF}.

\begin{proof}[Proof of Lemma \ref{LemIOF}]
The strategy is similar to De Silva-Savin \cite{DS}. We only give a brief sketch. 

For $A$ and $r_0$ to be chosen, suppose that the statement is false, then we find a sequence $\eps_k\to0$ and a sequence of solutions $w_k$, satisfying the hypothesis of the lemma, but for which the improvement is not achieved.

With the Harnack inequality in Theorem 5.1 of De Silva-Savin \cite{DS}, the normalized solutions 
$$
\tilde{w}_k:=\frac{1}{\eps_k}(w_k-x_d^+/\beta)
$$
converge to some $\tilde{w}$, solving 
$$
\begin{cases}
\Delta\tilde{w}+\gamma\frac{\partial_d\tilde{w}}{x_d}=0 &\text{ in }B_{3/4}\cap\{x_d>0\},\\
\partial_d\tilde{w}=0 &\text{ on }B_{3/4}\cap\{x_d=0\}.
\end{cases}
$$
This is achieved with arguments in  \textit{Step 2} from the proof of Proposition 6.1 in \cite{DS}. Compared with theirs, the drift term in our equation is purely in the $e_d$-direction. This is the main reason why we get an improvement in the H\"older exponent $\alpha$.

For this linearized equation, we argue as in the proof of Theorem 7.2 of \cite{DS}. 

We apply Theorem 7.7 in \cite{DS}  to see that 
$$
|\partial^2_{kk}\tilde{w}|\le C(\gamma, d)\hem \text{ in }B_{1/2}\cap\{x_d>0\}\text{ for all }k\le d-1.
$$
If we define, for a fixed $x'$, 
$$
v(t):=\tilde{w}(x',t),
$$
then $v$ solves
$$
v''+\gamma v'/t=-\sum_{k}\partial^2_{kk}\tilde{w}(x',t)\hem \text{ for }t>0.
$$
Such a function is of the form 
$$
v=c_1t^{1-\gamma}+c_2+f(t),
$$
where $f$ satisfies $$|f|\le Ct^2.$$

With the initial condition at $t=0$, we have the following expansion:
$$
|\tilde{w}(x',x_d)-\tilde{w}(0,0)-\nabla'\tilde{w}(0,0)\cdot x'|\le C(\gamma,d)r^2\hem \text{ in }B_r\cap\{x_d>0\}.
$$
Back to the original sequence $w_k$, this implies
$$
|w_k-x_d^+/\beta-\eps_k \nabla'\tilde{w}(0,0)\cdot x'|\le C(\gamma,d)\eps_k r^2+\eps_ko(1)\hem \text{ in }B_r \text{ as }k\to\infty. 
$$

\vem

We choose $A>0$ and $r_0>0$ such that 
$$
A=2C(\gamma, d), \hem Ar_0<1/2 \hem \text{ and }\log_{r_0}A>\alpha-1.
$$
Once $A$ and $r_0$ are fixed,  for $k$ large such that  $\eps_ko(1)<\eps_k C(\gamma,d)r_0^2,$ we have
$$
|w_k-x_d^+/\beta-\eps_k \nabla'\tilde{w}(0,0)\cdot x'|\le A\eps_k r_0^2\hem \text{ in }B_{r_0}.
$$
This gives the desired improvement. 
\end{proof}

With Proposition \ref{PropSharpAlpha}, we have the following expansion at the level of the gradient:
\begin{cor}
\label{CorAppExpansionOfGradient}
Under the same assumptions as in Proposition \ref{PropSharpAlpha}, we have
$$
|\nabla w(re)-e/\beta|\le C\eps r^\alpha \text{ for }r\in(0,1/4).
$$
\end{cor}
\begin{proof}
For $r\in(0,1/4)$, define
$$
w_r(x):=\frac{1}{r}w(rx+re), \text{ and }v(x):=(x+e)\cdot e/\beta.
$$
For small $\eps>0$, Proposition \ref{PropSharpAlpha} implies that both $w_r$ and $v$ solve
$$
\Delta u=\frac{\frac{\gamma}{2\beta}-(\beta-1)|\nabla u|^2}{u}\hem\text{ and } \hem u\ge\frac{1}{2\beta} \text{ in }B_{1/4}.
$$
Moreover, we have
$$
|w_r-v|\le C\eps r^\alpha \text{ in }B_{1/4}.
$$

Studying the equation for $(w_r-v)$, we apply standard elliptic estimate to conclude
$$
|\nabla w_r(0)-e/\beta|\le C\eps r^\alpha,
$$
which is equivalent to the desired estimate. 
\end{proof}


\end{document}